\documentclass[11pt,reqno]{article}

\usepackage[margin=1in]{geometry}
\usepackage{amsmath,amsthm,amssymb}

\usepackage{import}

\usepackage[colorlinks=true, pdfstartview=FitV, linkcolor=blue,citecolor=blue, urlcolor=blue]{hyperref}

\usepackage[abbrev,lite,nobysame]{amsrefs}
\usepackage{times}
\usepackage[usenames,dvipsnames]{color}

\usepackage{mathtools,enumitem}

\usepackage[compact]{titlesec}
 
\usepackage[title]{appendix}

\usepackage{comment}

\mathtoolsset{showonlyrefs=true}
    
\newcommand{\eps}{\epsilon}

\newcommand{\grad}{\nabla}

\newcommand{\norm}[1]{\left\| #1 \right\|}
\newcommand{\abs}[1]{\left| #1 \right|}
\newcommand{\set}[1]{\left\{ #1 \right\}}
\newcommand{\brak}[1]{\left\langle #1 \right\rangle} 

\newcommand{\nchoosek}[2]{ \begin{pmatrix} #1 \\ #2 \end{pmatrix} } 

\newcommand{\R}{\mathbb{R}}

\newcommand{\Z}{\mathbb{Z}}

\newcommand{\loc}{\mathrm{loc}}

\newcommand{\cL}{\mathcal{L}}
\newcommand{\cC}{\mathcal{C}}

\newcommand{\cM}{\mathcal{M}}

\newcommand{\cD}{\mathcal{D}}

\newcommand{\cT}{\mathcal{T}}

\newcommand{\Xbf}{\mathbf{X}}

\newcommand{\dee}{\mathrm{d}}
\newcommand{\ds}{\dee s}
\newcommand{\dt}{\dee t}
\newcommand{\dv}{\dee v}

\newcommand{\dx}{\dee x}

\newcommand{\brv}{\left<v\right>}

\DeclareMathOperator{\Div}{\mathrm{div}}

\newcommand{\sev}{\sigma_{k}e_k\cdot\nabla_v}
\newcommand{\pab}{\partial_{x}^{\alpha} \partial_{v}^\beta}
\newcommand{\pvab}{\partial_{x}^{\alpha} \partial_{v}^{\beta+e_j}}
\newcommand{\pxab}{\partial_{x}^{\alpha+e_j} \partial_{v}^{\beta}}

\newcommand{\cd}{\mathbb{D}}
\newcommand{\cN}{\mathcal{N}}

\usepackage{bbm}

\newcommand{\F}{\mathcal{F}}
\renewcommand{\P}{\mathbf{P}}

\newcommand{\E}{\mathbf{E}}
\newcommand{\EE}{\mathbf E}
\newcommand{\PP}{\mathbf P}

\newtheorem{theorem}{Theorem}[section]
\newtheorem{proposition}[theorem]{Proposition}
\newtheorem{corollary}[theorem]{Corollary}
\newtheorem{lemma}[theorem]{Lemma}
\newtheorem*{lemma*}{Lemma}

\theoremstyle{definition}
\newtheorem{definition}[theorem]{Definition}
\newtheorem{remark}[theorem]{Remark}

\usepackage{dirtytalk}

\allowdisplaybreaks
\numberwithin{equation}{section}
    
\begin{document}

\title{The Vlasov-Poisson and Vlasov-Poisson-Fokker-Planck systems in stochastic electromagnetic fields: local well-posedness} 
\author{Jacob Bedrossian\thanks{Department of Mathematics, University of Maryland, College Park, MD 20742, USA \href{mailto:jacob@math.umd.edu}{jacob@math.umd.edu}. Both J.B. and S.P. were supported by National Science Foundation Award DMS-2108633.} \and Stavros Papathanasiou\thanks{Department of Mathematics, University of Maryland, College Park, MD 20742, USA \href{mailto:stavrosp@umd.edu}{stavrosp@math.umd.edu}}}

\maketitle
\begin{abstract}
In this paper, we construct unique, local-in-time strong solutions to the Vlasov-Poisson (VP) and Vlasov-Poisson-Fokker-Planck (VPFP) systems subjected to external, spatially regular, white-in-time electromagnetic fields in $\mathbb T^d \times \mathbb R^d$. Initial conditions are taken $H^\sigma$ with $\sigma > d/2 + 1$ (in addition to polynomial velocity weights).  
We additionally show that solutions to the VPFP are instantly $C^\infty_{x,v}$ due to hypoelliptic regularization if the external force fields are smooth. 
The external forcing arises in the kinetic equation as a stochastic transport in velocity, which means, together with the anisotropy between $x$ and $v$ in the nonlinearity, that the local theory is a little more complicated than comparable fluid mechanics equations subjected to either additive stochastic forcing or stochastic transport. 
Although stochastic electromagnetic fields are often discussed in the plasma physics literature, to our knowledge, this is the first mathematical study of strong solutions to nonlinear stochastic kinetic equations.
\end{abstract} 

\setcounter{tocdepth}{1}
{\small\tableofcontents}

\section{Introduction}\label{sec:intro}
In this paper we prove the local-in-time existence and uniqueness of (probabilistically strong) solutions of the Vlasov and Vlasov-Fokker-Planck equations for a distribution of charged particles subjected to a stochastic external electric field
\begin{align}
  \begin{cases}\dee f + v\cdot\nabla_xf\dt + E\cdot\nabla_vf\dt + \nabla_vf \odot \dee W_t = \nu \nabla_v\cdot(\nabla_vf + fv)\dt\\
\rho = \int f\dv\\
E = \nabla_x(-\Delta_x)^{-1}(\rho-1) \\
\rho(t,x) = \int_{\mathbb R^d} f(t,x,v) \dee v,  
  \end{cases}
  \label{SVPFP}
\end{align}
where $\nu \geq 0$ is the collision frequency -- we treat both the case $\nu > 0$ and $\nu = 0$ (i.e. the Vlasov--Poisson equations).
Below we denote the Fokker-Planck operator
$$\mathcal L f  := \Delta_v f + \grad_v \cdot(v f),$$
which is a commonly used simplification for collisions of charged particles against a background (see e.g. \cite{boyd2003physics}).  

Here, we consider the problem in the periodic box $(x,v) \in \mathbb{T}^d\times\mathbb{R}^d$, although the case $(x,v) \in \mathbb R^d \times \mathbb R^d$ could be approached with similar arguments.
The process $W_t$ is a white-in-time, colored-in-space, vector-valued Gaussian process which plays the role of an external fluctuating electric field which we describe in more detail below.
Our analysis works for general $1 \leq d \leq 3$ and also applies to external magnetic fields. 
For simplicity, we will take initial conditions in the velocity-weighted Sobolev space $H^{\sigma}_m$ defined by the norm:
$$\norm{f}_{H^\sigma_m}^2 := \sum_{|\alpha|+|\beta|\leq \sigma}\iint_{\mathbb T^d \times \mathbb R^d} |\partial_x^\alpha\partial_v^\beta f(x,v)|^2 (1+|v|^2)^{m/2} \dv\dx.$$

Stochastic and randomly fluctuating electromagnetic fields are a classical topic in the plasma physics literature where they are used as a model for studying various dynamics in ``turbulent''-like situations in both confined fusion and astrophysical applications. 
Much of the work is on studying the motion of charged particles (i.e. Lagrangian trajectories or passive scalars) subjected to stochastic electromagnetic fields of various kinds; see e.g. \cites{hall1967diffusion,hasselmann1968scattering,jaekel1992fokker,hall1969particle,balescu1994langevin,krommes1983plasma,wingen2006influence,wang1995diffusive,vanden1996statistical} and the references therein for a tiny fraction of the existing work on the topic.
Another line of work regards subjecting gyrokinetic equations or other macroscopic models to randomly fluctuating external force fields of this type for the purpose of studying plasma turbulence; see e.g. \cites{tenbarge2014oscillating,navarro2016structure,told2015multiscale} and the references therein. 
The purpose of this work is to begin laying down some rigorous mathematical theory for studying nonlinear kinetic theory models of plasmas in these kinds of settings. 

It is sometimes useful to make a concrete representation of $W_t$ and for simplicity we will show how to do this in $d=3$; the extension to other dimensions is straightforward and is omitted.
To make this concrete representation of $W_t$, we define a real Fourier basis on $L^2(\mathbb T^3;\mathbb R^3)$ by defining for each $k = (\ell,i) \in \mathbb{K} := \Z^d_0 \times \{1,2,3\}$
 \[\label{eq:Fourier-Basis}
 e_k(x) = \begin{cases}
 c_d\gamma_\ell^i\sin(\ell \cdot x), \quad& \ell \in \Z^d_+\\
 c_d\gamma_\ell^i\cos(\ell \cdot x),\quad& \ell\in \Z^d_-,
 \end{cases}
 \]
where $\Z^d_0 := \Z^d \setminus \set{0,\ldots, 0}$, $\Z_+^d = \{\ell \in \Z^d_0 : \ell^{(d)} >0\}\cup\{\ell \in \Z^d_0 \,:\, \ell^{(1)}>0, \ell^{(d)}=0\}$ and $\Z_-^d = - \Z_+^d$, and for each $\ell \in \Z_0^d$, $\{\gamma_\ell^i\}_{i=1}^{3}$ is a set of three orthonormal vectors with $\set{\gamma_\ell^{1},\gamma_\ell^2}$ spanning the plane perpendicular to $\ell \in \R^3$ with the property that $\gamma_{-\ell}^i = - \gamma_{\ell}^i$ and $\gamma_\ell^3$ parallel to $\ell$. 
The constant $c_d = \sqrt{2}(2\pi)^{-d/2}$ is a normalization factor so that $e_k(x)$ are a complete orthonormal basis on $L^2$.
With this, we define our external electric field as 
\begin{align*}
W_t(x) = \sum_{k \in \mathbb{K}} \sigma_k e_k(x) W_t^{(k)},
\end{align*}
with $\set{ W_t^{(k)} }_{k \in \mathbb K}$ is a family of independent standard Wiener processes defined on a given stochastic basis $(\Omega, \mathcal{F}, \mathcal{F}_t, \P)$.
The $\sigma_{k}$ are coloring coefficients satisfying at least $\sum_{k \in \mathbb K} \abs{\sigma_{k}}^2 < \infty$, however, more stringent regularity requirements will be assumed below (here we make the natural definition $\abs{k} = \abs{\ell}$ for $k = (\ell,i) \in \mathbb K$). 
We can also treat the case of fluctuating magnetic fields; see Remark \ref{rmk:SMF} below.  

Local well-posedness of strong solutions for the deterministic problem is classical; see e.g. \cites{horst1981classical,horst1987global} for the Vlasov equations and \cites{neunzert1984vlasov,victory1990classical} for the Vlasov-Fokker-Planck equation.
Global existence for the deterministic problems was proved in \cite{pfaffelmoser1992global} (see also \cites{schaeffer1991global,batt1991global}) for the Vlasov equations and \cite{bouchut1993existence} for the Vlasov-Fokker-Planck equations; we will consider global existence for \eqref{SVPFP} in a follow up work. 
Notice that in It\^o form, the SPDE becomes 
\begin{equation}
\dee f + v\cdot\nabla_xf \dt + E\cdot\nabla_vf\dt + \nabla_vf\cdot\dee W_t = \nu \mathcal{L}f\dt + \frac{1}{2}\sum_k (\sev)^2f \dt, 
\end{equation}
so it is clear  that stochastic transport cannot be treated perturbatively with respect to the deterministic evolution, as the Stratonovich-It\^o correction term is of second order. However, this correction term is subelliptic, and so stochastic transport enjoys a special structure that makes it possible to develop a strong well-posedness theory, and in fact, it is sometimes possible to produce a better well-posedness theory for stochastic transport than for  deterministic transport \cite{flandoli2010well}. 
Due to this special structure and the many physical applications, there have been a great number of works studying stochastic transport equations recently; see for example \cites{fedrizzi2011pathwise,fedrizzi2013noise,beck2019stochastic,mohammed2015sobolev,champagnat2018strong} and also \cites{fedrizzi2017regularity,de2018invariant} in the kinetic case. 

There have been many works on fluid equations subjected to multiplicative or transport-type stochastic forcing.
For the Navier-Stokes equations see for example \cites{brzezniak1992stochastic,mikulevicius2004stochastic,brzezniak2013existence,capinski1993navier,mikulevicius2005global}. 
The 2D Euler equations in vorticity form subjected to transport noise was studied in, for example \cites{brzezniak2016existence,crisan2019well}, where strong solutions with bounded vorticity were constructed (see also \cite{crisan2019solution}). The aforementioned papers \cites{crisan2019well,crisan2019solution} belong to the so-called theory of \say{Stochastic Advection by Lie Transport} (SALT), see the foundational paper \cite{holm2015variational}, as well as \cites{crisan2020local,alonso2020well}.
The work \cite{brzezniak2020well} studies the 3D primitive equations with transport noise. 

The works \cites{debussche2011local,debussche2012global}
provide a fairly general framework to study  a wide class of dissipative fluid equations forced with multiplicative noise, such as the Navier-Stokes equations or the primitive equations. For the 2D Euler equations with various types of general multiplicative noise, see \cite{GV14} and the references therein. 

In comparison to stochastic fluid dynamics, the work on nonlinear, stochastically forced kinetic equations is significantly thinner.
The paper \cite{punshon2018boltzmann} constructs global-in-time renormalized martingale (probabilistically weak) solutions of the Boltzmann equations with external stochastic forcing similar to that used in \eqref{SVPFP}. 
In work with a clear relationship with our own, \cite{delarue2014noise} constructs global solutions of interacting point charges (i.e. Vlasov--Poisson with solutions given by a finite number of Dirac masses) subjected to stochastic external electric fields; see also \cite{coghi2016propagation}.

In this paper we continue the study of stochastic kinetic theory by proving local existence and pathwise uniqueness of strong solutions.
Let us recall some standard notions for probabilistically strong solutions of SPDEs that may experience finite-time blow up (we follow the presentation used in \cites{debussche2012global,GV14}), which are nothing more than the natural stochastic analogues of the deterministic notions of local-in-time existence, uniqueness, and maximally-extended solutions.
\begin{definition}\label{local:pathwise}
A \emph{local pathwise} solution of \eqref{SVPFP} is a pair $(f,\tau)$ with $\tau$ an almost-surely strictly positive stopping time and $f$ an adapted stochastic process satisfying the regularity
\begin{align*}
f(\cdot \wedge \tau) \in C([0,\infty);H^\sigma_m)
\end{align*}
and for $t \geq 0$ satisfies, 
\begin{align*}
f(t\wedge \tau) - f(0) + \int_0^{t \wedge \tau} \left( v \cdot \grad_x f(s) + E(s)\cdot \grad_v f - \nu \mathcal{L} f(s) \right) d\tau = \int_0^{t \wedge \tau} \grad_v f(s) \circ \dee W_s.  
\end{align*}
Moreover, we say such pathwise solutions are \emph{unique} if for any pair $(f_1,\tau_1)$, $(f_2,\tau_2)$ we have
\begin{align*}
\mathbb P \left( f_1(t) - f_2(t) = 0 \quad \forall 0 \leq t < \tau_1 \wedge \tau_2 | f_1(0) = f_2(0) \right) = 1. 
\end{align*}
In this case, $f_1$ and $f_2$ are called \emph{indistinguishable}. 
\end{definition}
The following definition of maximal pathwise solution provides a continuation criterion for strong solutions.
For this we will use $H^{d/2+1 +}_{d/2+}$; sharper continuation criteria will be considered in future work.
That is, we show that local $H^\sigma_m$ solutions can be uniquely extended provided some $H^{s_0}_{m_0}$ norm remains finite for $s_0 > 1+d/2$ and $m_0 > d$ fixed and arbitrary.
\begin{definition} \label{def:MPS}
Fix $s_0 > d/2 +1$ an integer and $m_0 > d$. 
We call a \emph{maximal pathwise solution} a triple of a solution $f$, an increasing sequence of almost-surely positive stopping times $\set{\tau_n}_{n \geq 0}$, and a limiting stopping time $\xi$ such that each pair $(f,\tau_n)$ is a local pathwise solution, $\lim_{n \to \infty} \tau_n = \xi$, and
\begin{align*}
\sup_{ 0 \leq t \leq \tau_n} \norm{f(t)}_{H^{s_0}_{m_0}} \geq n \quad \textup{ on the set } \set{\xi < \infty}. 
\end{align*}
\end{definition}
In this paper, we prove the following local existence and uniqueness theorem.
Global existence of these strong solutions will be considered in a future work. 
\begin{theorem} \label{thm:main}
Let $1 \leq d \leq 3$.  
Let $\sigma > s_0$ and $m > m_0$ be fixed integers and assume that \begin{equation}
	\sum_{k \in \mathbb K} \abs{k}^{2\sigma'}\abs{\sigma_{k}}^2 < \infty \label{coloring}
\end{equation} 
for some $\sigma' > \sigma + 4$ (integer). 
Suppose that the initial condition $f_0$  is an $\mathcal{F}_0$-measurable random variable such that $f_0 \in H^\sigma_m$ almost-surely.
Then, there exists a unique, maximal pathwise solution to \eqref{SVPFP} for any $\nu \geq 0$.
\end{theorem}

\begin{remark} \label{rmk:SMF}
Our proof also applies when there is a stochastic magnetic field.
We may similarly treat the case of independent electric and magnetic fields as the following, for example (denoting $c > 0$ the speed of light), 
\begin{align*}
W_t(x) & \mapsto \sum_{k \in \mathbb K} \sigma_k^{(E)} e_k(x) W_t^{k;E} \\ & \quad + \frac{1}{c}v \times \sum_{\ell \in \mathbb Z^3_0} \sigma_\ell^{(B;1)} e_{(\ell,1)}(x) W_t^{(\ell,1);B} + \frac{1}{c}v \times \sum_{\ell \in \mathbb Z^3_0} \sigma_\ell^{(B;2)} e_{(\ell,2)}(x) W_t^{(\ell,2);B},
\end{align*}
with
\begin{align}
\sum_{k =(\ell,j) \in \mathbb K} \abs{k}^{2\sigma'}\left(\abs{\sigma_k^{(E)}}^2 +  \abs{\sigma_\ell^{(B;1)}}^2 + \abs{\sigma_\ell^{(B;2)}}^2\right) < \infty, \label{ineq:bkabs}
\end{align}
or when electromagnetic fields are correlated, for example one could use forcing of the following potentially natural form
\begin{align*} 
W_t & \mapsto \sum_{\ell \in \mathbb Z^3_0} \sigma_\ell^{(1)} \left( e_{(\ell,1)}(x) + \frac{v \times e_{(\ell,2)}(x)}{c} \right) W_t^{(\ell,1)}  + \sum_{\ell \in \mathbb Z^3_0} \sigma_\ell^{(2)} \left( e_{(\ell,2)}(x) + \frac{v \times e_{(\ell,1)}(x)}{c} \right) W_t^{(\ell,2)}. 
\end{align*}
Sufficiently regular-in-space deterministic external electromagnetic fields or random fields that are smoother in time than white noise (for example, Ornstein-Uhlenbeck processes and variations thereof as in the Langevin antenna forcing used in the plasma physics literature \cite{tenbarge2014oscillating}) can also be easily included in the analysis without any significant changes. For simplicity of presentation, we will mainly focus on the case of external electric fields and simply make comments about what changes when considering an additional magnetic field.
\end{remark} 

\begin{remark}
The methods of this paper can also deal with with more general mean-field interactions, replacing the self-consistent electric field with:
\begin{align*}
	E = \grad_x K \ast (\rho - 1), 
\end{align*}
for any kernel $K$ such that $\norm{\grad_x^{j+1} E}_{L^p} \lesssim_{p} \norm{\brak{\grad_x}^{j} \rho}_{L^p}$ for all $1 < p < \infty$.
\end{remark}

\begin{remark}
  It should be straightforward to extend to $d \geq 4$. 
  It should also be possible to treat non-integer $\sigma$, $s_0$, and $\sigma'$, however, this would require more delicate (anisotropic) commutator estimates. 
\end{remark}

\begin{remark} \label{rmk:Landau}
  We believe our methods could be extended to the Landau collision operators for initial data $f_0$ sufficiently close to a global Maxwellian to prove local-in-time existence and uniqueness of strong solutions to e.g. the Vlasov--Poisson--Landau equations with stochastic external electromagnetic fields. 
  This extension may be considered in future work. 
\end{remark}

\begin{remark}
In light of the classical deterministic theory of bounded solutions of the Vlasov equations (see e.g. \cites{lions1991propagation,loeper2006uniqueness}), it is natural to expect an analogue of \cite{brzezniak2016existence} in kinetic theory.
Similarly, we expect local (and global) existence and uniqueness of the Vlasov-Fokker-Planck equations using only e.g. $f_{0} \in L^2_m$.
These extensions may be considered in future work. 
\end{remark}

Finally, in Section \ref{sec:HR} we present a proof of the following hypoelliptic regularization result.
This is proved using a time-dependent hypocoercivity norm in the spirit of \cite{dric2009hypocoercivity}.  
\begin{theorem} \label{thm:HR}
Let $f$ be a maximal pathwise solution to \eqref{SVPFP} as in Theorem \ref{thm:main}.
Suppose that for all $N$ there holds 
\begin{align*}
\abs{\sigma_{k}} \lesssim_N \abs{k}^{-N}. 
\end{align*}
Then if $\nu > 0$, then $f(t) \in C^\infty_{x,v}$ for all $t \in (0,\xi)$. 
\end{theorem}

\section{Outline} \label{sec:outline}
Let us outline the general idea of how to prove Theorem \ref{thm:main} and then provide the details in the main body of the text.
See Section \ref{sec:HR} for how to prove Theorem \ref{thm:HR}. 

As in e.g. \cites{debussche2012global,GV14}, we first construct solutions to \eqref{SVPFP} with smoother initial data $H^{\sigma'}_{m'}$ with $\sigma' > \sigma+4$ and $m' > m + 3$ (both integers) with trajectories in $L^\infty_{t,loc} H^{\sigma'}_{m'} \cap C_{t,loc} H^{\sigma}_m$.
This procedure is done in Section \ref{sec:iii}. 
Then we regularize the initial condition and pass to the limit to obtain solutions with initial data in $H^\sigma_m$ that take values in $C_{t,loc} H^{\sigma}_m$.
In addition to obtaining solutions with lower regularity, what is more important for many purposes, is that this constructs solutions which take values continuously in the highest regularity available.
This latter procedure is done in Section \ref{sec:iv}. 

To construct maximal pathwise solutions to \eqref{SVPFP} we first introduce a standard trick for regularizing the nonlinearity in a way which allows to close necessary probabilistic moment estimates. 
Consider a smooth nonnegative and nonincreasing function $\theta: [0,\infty) \to \mathbb{R}$ such that:
\begin{equation}
	\theta(x) = \begin{cases}
		1 & \text{ if } x\leq 1,\\ 
		0 & \text{ if } x \geq 2,\end{cases}
\end{equation}
and define:
\begin{equation}
	\theta_R(x) \equiv \theta\left(\frac{x}{R}\right).
\end{equation}
Then we define the regularized SPDE
\begin{equation}
\dee f + v\cdot\nabla_xf\dt + \theta_{R}(\norm{f}_{H^{s_0}_{m_0}}) E\cdot\nabla_vf\dt + \nabla_vf\odot\dee W_{t} = \nu \mathcal{L} f \dt, \label{SVPFP:reg}
\end{equation}
We show in Section \ref{sec:iii} that this SPDE admits global-in-time, unique, pathwise solutions (i.e. $\xi = \infty$ with probability $1$ in the definition of maximal solutions) starting from $H^{\sigma'}_{m'}$ initial conditions.
Specifically, we prove the following. 
\begin{lemma}\label{glbl:reg}
	Let $f_0$ be a $\mathcal{F}_0$-measurable random variable such that $\forall p \geq 2$, 
    \begin{align*}
\EE \norm{f_0}_{H^{\sigma'}_{m'}}^p < \infty. 
\end{align*}
    Then, there exists an  $f\in C([0,\infty); H_{m'-3}^{\sigma'-4})\cap L_{t,\text{loc}}^\infty([0,\infty);H_{m'-1}^{\sigma'-1})$ $\P$--a.s. which is a solution to \eqref{SVPFP:reg} in the sense that  
\begin{equation}
	f(t) = f_0 + \int_0^t\left(-v\cdot\nabla_xf(s)-\theta_{R}(\norm{f}_{H^{s_0}_{m_0}}) E(s)\cdot\nabla_vf(s)+\nu\cL f(s)\right)\ds-\int_0^t\nabla_vf(s)\circ\dee W_s. \quad \P\text{--a.s.}, 
\end{equation}
where the equality holds in $C([0,\infty);H^{\sigma'-4}_{m'-3})$.
Moreover, if $\tilde{f}$ is any other solution in the above sense, then $f = \tilde{f}$ almost surely in the sense that
\begin{align*}
\mathbb P \left( f(t) - \tilde{f}(t) = 0 \quad \forall 0 \leq t < \infty | f(0) = \tilde{f}(0) \right) = 1. 
\end{align*}
\end{lemma}
It is clear that solutions of \eqref{SVPFP:reg} are also solutions to \eqref{SVPFP} for as long as $\norm{f}_{H_{m_0}^{s_0}} < R$, and so by considering the increasing sequence of $R = n$
and defining the stopping times
\begin{align*}
\tau_n = \inf \set{ t \geq 0: \norm{f(t)}_{H^{s_0}_{m_0}} > n }, 
\end{align*}
we may use \eqref{SVPFP:reg} to construct local pathwise solutions to \eqref{SVPFP}.
A standard cutting procedure (described below) also shows how to remove the moment requirement on the initial data.  
\begin{lemma}\label{lem:locHiRegSVPFP}
Let $f_0$ be an $\mathcal{F}_0$-measurable random variable with $f_0 \in H^{\sigma'}_{m'}$ almost surely.
Then, Lemma \ref{glbl:reg} implies the pathwise existence and uniqueness of a maximal solution $(f,\tau)$ to \eqref{SVPFP} with initial data $f_0$ with trajectories $f$ satisfying
\begin{align*}
f(\cdot \wedge \tau) \in L^\infty_{t,loc}(0,\infty;H^{\sigma'-1}_{m'-1}) \cap C_{t,loc}([0,\infty);H^{\sigma'-4}_{m'-3}). 
\end{align*}
\end{lemma}
\begin{proof}
First consider the case that $\norm{f_0}_{H^{\sigma'}_{m'}} < M$ almost-surely.
Then, we choose $R=M+1$ in \eqref{SVPFP:reg}, and define the stopping time:
\begin{equation*}
	\tau = \inf\{t\geq 0: \, \norm{f(t)}_{H_{m}^{\sigma}} > R\},
\end{equation*}
where $f$ is the solution to \eqref{SVPFP:reg} with initial data $f_0$, guaranteed to exist and be unique from Lemma \ref{glbl:reg}. Note that up to time $\tau$, the process $f$ also solves \eqref{SVPFP}, since for $t<\tau$ we have $\|f(t)\|_{H_{m_0}^{s_0}} \leq \|f(t)\|_{H_m^\sigma} \leq R$ and therefore $\theta_R(\|f(t)\|_{H_{m_0}^{s_0}})=1.$
Clearly $\tau >0$ almost surely since $H_{m'-2}^{\sigma'-2} \subset H_{m}^{\sigma}$ and $f$ takes values continuously in $H_{m'-2}^{\sigma'-2}$. The pair $(f,\tau)$ is thus a local solution of \eqref{SVPFP} within the higher regularity framework of this lemma, which is unique by Lemma \ref{uniqueness:reg} below. Now we will extend $f$ to a maximal solution.

Let $\cT$ be the collection of all stopping times corresponding to a local solution and define $\xi= \sup \cT$. Define also:
\begin{equation*}
	\tau_n := \inf\{t\geq 0: \, \norm{f(t)}_{H_{m_0}^{s_0}}>n\}.
\end{equation*}
Fix $T>0$ finite but arbitrary and assume $\P(\xi = \tau_n\wedge T)>0$ for some $n$. This implies that
\begin{align*}
\sup_{t\leq \xi}\norm{f(t)}_{H_{m_0}^{s_0}}\leq n \textup{ on the set } \{\xi=\tau_n\wedge T\},
\end{align*}
and thus $f$ can be continued to a solution of \eqref{SVPFP:reg} with $R=n+1$ and thus of \eqref{SVPFP} up to a stopping time past $\xi$ - contradicting $\xi$'s maximality. Since $T$ was arbitrary, we either have $\xi = \infty$, or $\tau_n<\xi<\infty$ for all $n$. In the latter case, we also get $\sup_{t< \xi}\norm{f(t)}_{H_{m_0}^{s_0}}\geq n$ for all $n$ and thus $\sup_{t< \xi}\norm{f(t)}_{H_{m_0}^{s_0}}=\infty.$

Now we drop the almost-sure uniform boundedness requirement. 
If $\norm{f_0}_{H_{m'}^{\sigma'}}<\infty$ almost surely, we decompose $f_0=\sum_{k=0}^\infty f_{0,k},$ where $f_{0,k} := \mathbbm{1}_{\{k\leq \norm{f_0}_{H_{m'}^{\sigma'}} < k+1\}} f_0$. Now each $f_{0,k}$ generates a maximal solution $(f_k,\tau_k)$ where $\tau_k$ is the corresponding maximal existence time, and we define the \say{total} maximal solution (in high regularity) of \eqref{SVPFP} as $(\bar{f},\bar{\tau})$ with:
\begin{gather} \label{eq:recon}
	\bar{f} = \sum_{k=0}^\infty \mathbbm{1}_{\{k\leq \norm{f_0}_{H_{m'}^{\sigma'}}<k+1\}}(\omega)f_k,\\
	\bar{\tau} = \sum_{k=0}^\infty \mathbbm{1}_{\{k\leq \norm{f_0}_{H_{m'}^{\sigma'}}< k+1\}}(\omega)\tau_k.
\end{gather}
\end{proof} 

Solutions to \eqref{SVPFP:reg} are constructed using a two-step procedure.
First, we regularize the nonlinearity again and use an iteration procedure to construct a solution to the regularized SPDE and then second, we pass to the limit in the additional regularization parameter.
Let $\varphi \in C^\infty_c(B(0,2))$ with $\int_{\mathbb R^n} \varphi \dee x = 1$ and define $\varphi_\eps = \eps^{-d} \varphi(\eps^{-1} \cdot)$.
Specifically, we seek a solution to the following regularized SPDE (here the convolution in $x$ has been periodized), 
\begin{align}
	& \dee \tilde f + v\cdot\nabla_x\tilde f\dt + \nabla_v\tilde f\circ \dee W_t + \theta_R(\norm{\tilde f}_{H_{m_0}^{s_0}}) (\varphi_\eps \ast \tilde E) \cdot\nabla_v\tilde f = \nu\cL \tilde f\dt   \label{SVPFP:reg:reg} \\
    & \tilde{E} = \grad_x (-\Delta_x)^{-1}\left( \int_{\mathbb R^d} \tilde f(t,\cdot,v) \dee v - 1 \right).  \notag
\end{align}
This is done by an iteration method, specifically the following 
\begin{align}
& f^j(0) = f_0 \\ 
& \dee f^{0} + v\cdot\nabla_xf^{0}\dt + \nabla_vf^{0}\odot\dee W_{t} = \nu \mathcal{L} f^{0}\dt \label{SVPFP:Step0iter}\\ 
& \dee f^{j+1} + v\cdot\nabla_xf^{j+1}\dt + \theta_R(\norm{f^j}_{H^{s_0}_{m_0}})(\varphi_\eps \ast E[f^j]) \cdot\nabla_vf^{j+1}\dt + \nabla_vf^{j+1}\odot\dee W_{t} = \nu \mathcal{L} f^{j+1}\dt, \label{SVPFP:iter}
\end{align}
where we denote
\begin{align*}
E[f] := \grad_x (-\Delta_x)^{-1}\left( \int_{\mathbb R^d} f(t,\cdot,v)  \dee v - 1 \right). 
\end{align*}
The solutions to \eqref{SVPFP:Step0iter} - \eqref{SVPFP:iter} are constructed by the method of characteristics.
Indeed, \eqref{SVPFP:iter} is the forward Kolmogorov equation associated to the SDE
\begin{equation}
\begin{cases}\dee X_t = V_t\dt\\
\dee V_t = -\nu V_t \dt + \theta_R(\norm{f^j}_{H^{s_0}_{m_0}}) (\varphi_\eps \ast E[f^j])(t,X_t) \dt  + \sqrt{2 \nu }\dee \tilde{W}_t + \sum_{k}\sigma_{k}e_{k}(X_t) \circ \dee W_t^{(k)}, 
\end{cases}\label{schar}
\end{equation}
where $(\tilde{W}_t)$ is a $d$-dimensional Brownian motion defined in a new stochastic basis $(\Omega',\mathcal{F}', \P')$ (independent of the original basis). 
That is, \eqref{schar} are the \textit{stochastic characteristics} corresponding to \eqref{SVPFP:iter}, which generates a global stochastic flow of volume-preserving diffeomorphisms $\phi_{t}$ on $\mathbb{T}^d\times\mathbb{R}^d$, defined on the product space
$$(\Omega \times \Omega', \mathcal{F}\otimes \mathcal{F}', \P\times \P'),$$ which map $H^{\sigma'}_{m'}$ back to itself for all finite times almost-surely. 
The multiplicative (linear!) SPDE \eqref{SVPFP:iter} is then solved by a \say{partial Feynman-Kac} formula with respect to $\P':$
\begin{equation}
 f^{j+1} = \E_{\P'} f_0 \circ \phi_{t}^{-1}. \label{FK}
\end{equation}
See \cite{kunita1997stochastic} for more details.
\begin{remark}
Note that this type of regularization procedure has the added benefit of retaining non-negativity of $f$ as well as the preservation of the Casimir conservation laws, e.g. if $\nu = 0$ then $\norm{f^j}_{L^p} = \norm{f_0}_{L^p}$ and for $\nu > 0$ one at least has $\norm{f^j}_{L^p} \leq e^{d \nu t}\norm{f_0}_{L^p}$.
However, these properties do not play an important role here. 
\end{remark}
Next, we need uniform a priori estimates to enable passing $j \to \infty$.
These are obtained via Eulerian energy methods and come out as $\forall p < \infty$, $\alpha \in (0,1/2)$, and $T <\infty$, 
\begin{align*}
& \sup_{j \geq 1} \EE \norm{f^j}_{L^\infty(0,T;H^{\sigma'}_{m'})}^p  \lesssim_{p,T,\eps} 1 \\ 
& \sup_{j \geq 1} \EE \norm{f^j}_{W^{\alpha,p}(0,T;H^{\sigma'-2}_{m'-1})}^p \lesssim_{p,T,\eps,\alpha} 1. 
\end{align*}
See Lemma \ref{comp:it} for the proof of these estimates.
Several previous works, for example \cites{debussche2012global,GV14,brzezniak2020well} have used compactness to pass to similar limits, extract martingale solutions (i.e. probabilistically weak) using the Skorokhod embedding theorem, and then subsequently upgrade these solutions using a Gy\"ongy-Krylov lemma \cite{gyongy1996existence} argument and pathwise uniqueness.
However, this technique seems not to apply in a clear manner to the Lagrangian iteration \eqref{SVPFP:iter}.
Instead, we prove directly that there is a stopping time $\xi$ which is almost-surely greater than $1$ such that $\set{f^j}_{j \geq 0}$ forms a Cauchy sequence in $L^2(\Omega;C_{t}([0,\xi);H^{s_0}_{m_0}))$, at which point it is not hard to pass to the limit, iterate in $t$, and construct global solutions to \eqref{SVPFP:reg:reg} in the desired regularity classes.
This is proved in Lemma \ref{iter:Cauchy}, where, in analogy with a classical Picard iteration, we show that $f^{j+1}-f^j$ is nearly comparable in size to the $j$-th term of a power series in powers of $\sqrt{t}$ of the solution.
This procedure finally yields 
\begin{lemma}\label{glbl:reg:reg}
	Let $f_0$ be an $\F_0$-measurable random variable such that $\forall p \geq 2$, 
    \begin{align*}
\EE \norm{f_0}_{H^{\sigma'}_{m'}}^p < \infty. 
\end{align*}
    Then, there exists an $f\in C([0,\infty); H_{m'-2}^{\sigma'-3})\cap L_{t,\text{loc}}^\infty([0,\infty);H_{m'}^{\sigma'})$ $\P$--a.s. which is a solution to \eqref{SVPFP:reg:reg} in the sense that , $\P$--a.s.: 
\begin{equation}
	f(t) = f_0 + \int_0^t\left(-v\cdot\nabla_xf(s)-\theta_{R}(\norm{f}_{H^{s_0}_{m_0}}) \varphi_\eps*E(s)\cdot\nabla_vf(s)+\nu\cL f(s)\right)\ds-\int_0^t\nabla_vf(s)\circ\dee W_s, 
\end{equation}
where the equality holds in $C([0,\infty);H^{\sigma'-3}_{m'-2})$.
Moreover, if $\tilde{f}$ is any other solution in the above sense, then $f = \tilde{f}$ almost surely in the sense that
\begin{align*}
\mathbb P \left( f(t) - \tilde{f}(t) = 0 \quad \forall 0 \leq t < \infty | f(0) = \tilde{f}(0) \right) = 1. 
\end{align*}
\end{lemma}

The next step in the proof of Lemma \ref{glbl:reg} is to remove the  superfluous mollifier $\varphi_\eps$, which  begins with obtaining $\eps$-independent estimates (now indexing solutions to \eqref{SVPFP:reg:reg} by $\eps$), 
\begin{align*}
& \sup_{\eps \in (0,1)} \EE \norm{f_{\eps}}_{L^\infty(0,T;H^{\sigma'}_{m'})}^p  \lesssim_{p,T,R} 1 \\ 
& \sup_{\eps \in (0,1)} \EE \norm{f_{\eps}}_{W^{\alpha,p}(0,T;H^{\sigma'-2}_{m'-2})}^p \lesssim_{p,T,R,\alpha} 1. 
\end{align*}
See Lemma \ref{comp:epsilon} for the proof of these estimates.
These estimates can be considered the probabilistic analogue of the common deterministic method of sharpening a continuation criterion a posteriori, specifically, the thrust of the estimates is to show that the $H^{s_0}_{m_0}$ norm controls all $H^{\sigma'}_{m'}$ norms for $m' > m_0$ and $\sigma' > m_0$. 
At this step, it does not seem straightforward to prove that $\set{f_\eps}_{\eps \in (0,1)}$ is Cauchy, and so we follow the martingale approach.
Specifically, we use these uniform bounds to apply the Skorokhod embedding theorem to produce probabilistically weak solutions to \eqref{SVPFP:reg} (see Proposition \ref{martingale:reg} below). 
These solutions are subsequently upgraded to probabilistically strong solutions by proving pathwise uniqueness (Lemma \ref{uniqueness:reg}) and an application of the Gy\"ongy-Krylov lemma (from \cite{gyongy1996existence}; see Lemma \ref{GK} below). 
This general procedure is rather standard at this point; see for example \cites{debussche2012global,GV14,brzezniak2020well}. 
This step completes the proof of Lemma \ref{glbl:reg}. 

The final step in the proof of Theorem \ref{thm:main}  is to pass to a suitable limit in order to construct solutions in $C_t H^\sigma_m$ from $H^\sigma_m$ initial data, which is done in Section \ref{sec:iv}. 
We perform a regularization procedure on the initial data by defining a sequence of initial conditions
\begin{align*}
f_{0;n} = \mathcal{R}^n f_0 \equiv \theta_n(v) n^{2d} \eta\left( \frac{\cdot}{n} \right) \ast_{x,v} f_0,
\end{align*}
where $\eta \in C^\infty_c(\mathbb R^{2d})$ and satisfies $\eta \geq 0$ and $\int_{\mathbb R^{2d}} \eta \dee x \dee v= 1$.
Note these have been both mollified and cut-off in velocity (to improve both regularity and localization). 
For all $f_0 \in H_m^\sigma \cap L^1_+$, we hence have $f_{0;n} \in H^{\sigma'}_{m'} \cap L^1_+$ for all $\sigma',m'$.
Subsequently, there are unique local pathwise solutions to \eqref{SVPFP} $(f_n,\tau_n)$ with
\begin{align*}
f_n(\cdot \wedge \tau_n) \in L^\infty_{loc}(0,\infty;H^{\sigma'}_{m'}) \cap C([0,\infty); H^\sigma_m). 
\end{align*}
By obtaining suitable uniform-in-$n$ upper bounds on the $C_t H^\sigma_m$ norm, we may pass to the limit $n \to \infty$ and hence extract local pathwise solutions to the original problem in $C_t H^\sigma_m$; see Lemmas \ref{lemma:ACL} and \ref{ACL:ver} for details.

\subsection*{Notation and conventions}
For technical reasons, it is sometimes necessary (particularly when passing to the limit in the proofs of Lemmas \ref{martingale:reg} and \ref{glbl:reg}) to view the fluctuating field as coming from a \textit{cylindrical} Wiener process. Specifically, let $\mathfrak{U}$ be a separable Hilbert space, with an orthonormal basis $(g_k)_{k\in\mathbb{K}}.$ We can formally define a cylindrical Wiener process $\mathcal{W}_t$ on $\mathfrak{U}$ by the formula \begin{equation*}
	\mathcal{W}_t := \sum_{k\in\mathbb{K}} g_k W_t^{(k)}.
\end{equation*}
Since this sum is divergent on $\mathfrak{U},$ one frequently employs the larger Hilbert space:
\begin{gather*}
	\mathfrak{U}_0 := \left\{\sum_k \alpha_k g_k : \sum_k k^{-2}\alpha_k^2<\infty \right\},\\
	\left\|\sum_k\alpha_kg_k\right\|_{\mathfrak{U}_0}^2 = \sum_k k^{-2}\alpha_k^2,
\end{gather*}
where it can be shown that the formal sum for $\mathcal{W}_t$ converges and defines a process whose paths are almost surely in $C([0,T);\mathfrak{U}_0)$. Moreover, the embedding $\mathfrak{U} \subset \mathfrak{U}_0$ is Hilbert--Schmidt. For any separable Hilbert space $X$, we denote the space of all Hilbert--Schmidt operators from $\mathfrak{U}$ to $X$ by $L_2(\mathfrak{U};X)$; the definition of this space is:
\begin{gather*}
	L_2(\mathfrak{U};X) := \left\{T\in L(\mathfrak{U};X): \sum_k \|Tg_k\|_{X}^2 <\infty\right\},\\
	\norm{T}_{L_2(\mathfrak{U};X)}^2 = \sum_k \|Tg_k\|_{X}^2.
\end{gather*} 
For more details on cylindrical Wiener processes and the relevant functional analytic setting, see \cite{da2014stochastic}. 

At various points, we use the notation $f \in L^{p-}$ to signify that $f$ is in any $L^q$ space for $q<p$.

We often employ the common notation:
\begin{equation*}
	A(f) \lesssim_{p_1,p_2,\dots} B(f)
\end{equation*} 
which means that there exists a constant $C>0$ depending only on the parameters $p_1,p_2,\dots$ but \textit{not} on the argument $f$, such that $A(f) \leq CB(f)$ for all relevant $f$. We omit the parameters if they are unimportant or clear from the context.

For the velocity-weighted $L^2$ norms and inner products, we set:
\begin{gather*}
	\brak{f,g}_m := \iint_{\mathbb T^d \times \mathbb R^d}f(x,v)g(x,v)\brv^m\dv\dx\\
	\|f\|_{L_m^2}^2 := \brak{f,f}_m.
\end{gather*}

Finally, at various points we use the mixed weighted norms:
\begin{gather*}
	\|f\|_{L_{v,n}^2L_x^p} := \left(\int_{\mathbb{R}^d}\left[\int_{\mathbb{T}^d}|f(x,v)|^p\dx\right]^{2/p} \brv^n\dv\right)^{1/2},\notag\\
	\|f\|_{L_{v,n}^2H_x^s} := \left(\sum_{|\alpha|\leq s}\int_{\mathbb{R}^d}\int_{\mathbb{T}^d} |\partial_x^\alpha f(x,v)|^2\brv^n\dx\dv\right)^{1/2}.
\end{gather*}
\section{Very smooth solutions and pathwise uniqueness} \label{sec:iii}
\subsection{Proof of Lemma \ref{glbl:reg:reg}}

As discussed in Section \ref{sec:outline}. a key step in proving Lemma \ref{glbl:reg:reg} is constructing a convergent sequence of approximate solutions derived from a Lagrangian iteration scheme for \eqref{SVPFP:reg:reg}. In particular, consider a sequence $f^j$ defined inductively as:
\begin{equation}
	\begin{cases}
		\dee f^0 + v\cdot\grad_xf^0\dt + \grad_v f^0\circ \dee W_t = \nu \cL f^0\dt,\\
		\dee f^{j+1} + v\cdot\grad_xf^{j+1}\dt + \grad_v f^{j+1}\circ \dee W_t = \nu \cL f^{j+1}\dt - \theta_R(\|f^j\|_{H_{m_0}^{s_0}})(\varphi_\eps \ast E^j)\cdot\grad_vf^{j+1}\dt,\\
		f^j(0) = f_0.
	\end{cases}\label{iteration}
\end{equation}
As discussed in Section \ref{sec:outline}, for a given $f^{j} \in f \in C([0,\infty); H_{m'-2}^{\sigma'-3})\cap L_{t,loc}^\infty([0,\infty),H_{m'}^{\sigma'})$, the solution $f^{j+1}$ is constructed via the method of stochastic characteristics. 

First, we provide $j$-independent estimates in order to pass to the limit $j \to \infty,$ for which we need appropriate compactness estimates for the iterates $f^j$ defined above.
We remark that studying the stochastic flow of diffeomorphisms could show that  $f^{j} \in f \in C([0,\infty); H_{m'-2}^{\sigma'-3})\cap L_{t,loc}^\infty([0,\infty),H_{m'}^{\sigma'})$, however, providing $j$-independent (and eventually $\eps$-independent) bounds seems to be significantly more complicated than an Eulerian energy method approach, which is hence the approach we take. 
The main ingredient is provided by the following lemma:
\begin{lemma}\label{comp:it}
  Let $(f^j)_{j\geq 1}$ be a sequence of global solutions to the iterative scheme \eqref{iteration} with $\E\norm{f_0}_{H^{\sigma'}_{m'}}^p < M_p <\infty$ for all $p\geq 2.$
    For $\alpha \in (0,\frac{1}{2}), p\geq 2,$ we have the uniform estimates:
	\begin{equation}\sup_{j\geq 1}\E\sup_{t\leq T} \|f^j(t)\|_{H_{m'}^{\sigma'}}^p \leq C_{T,R,\eps,M}\label{unif:bdd:it}
	\end{equation}
and
	\begin{equation}
	\sup_{j\geq 1}\E\|f^j\|_{W^{\alpha,p}([0,T]; H_{m'-1}^{\sigma'-2})}^p \leq C_{T,R,\eps,M}.\label{equic:it}
	\end{equation}
\end{lemma}



Before we begin, let us begin by recalling a few standard estimates.
The first shows how to estimate the density in terms of the distribution function using sufficiently many velocity  moments.
\begin{lemma} \label{interpolation}
For any $m > d$ there exists a constant $C_{m,d}$ such that
\begin{equation}
	\left\|\int f(x,v)\dv\right\|_{L_x^2} \leq C\|f\|_{L_m^2}
\end{equation}
\end{lemma}
Next, we recall the following Gagliardo-Nirenberg-Sobolev estimate: for all integers $0 \leq i \leq \sigma$ and functions $f \in H^\sigma$ (in $\mathbb T^d$ or $\mathbb R^n$) there holds
\begin{align}
\sum_{\abs{\alpha} = i} \norm{\partial^\alpha f}_{L^{\frac{2\sigma}{i}}} \leq C \norm{f}_{L^\infty}^{1 - \frac{i}{\sigma}} \left(\sum_{\abs{\alpha} = \sigma}\norm{\partial^\alpha f}_{L^2}\right)^{\frac{i}{\sigma}}. \label{ineq:iGNS}
\end{align}
The next estimate recalls how to adapt Sobolev space product rules to the anisotropic nonlinearity.
We give a proof for the readers' convenience. 
\begin{lemma}\label{prod:rules}
	Let $g\in H_{n_0}^{s-1},f\in H_{n}^{s}$ for some $n_0> d$, $s > \frac{d}{2}+1$, and $n \geq 0$ arbitrary.
    Denote $E^g := \grad_x\Delta_x^{-1}\left(\int g\dv -1 \right).$ 
	Then, for a constant $C$ that does not depend on $g$ or $f$:
	\begin{equation}
		\|E^g\cdot\grad_v f\|_{H_{n}^{s}}\leq C\|g\|_{H_{n_0}^{s-1}}\|\grad_vf\|_{H_{n}^{s}} \label{prod:rule:1}
	\end{equation} 
and
\begin{equation}
	\sum_{|\alpha|+|\beta|\leq s}\left<\pab(E^g\cdot\grad_vf),\pab f\right>_n \leq C(\|E^g\|_{W^{1,\infty}}+\|g\|_{H_{n_0}^{s-1}})\|f\|_{H_n^{s}}^2\label{prod:rule:2}.
\end{equation}
\end{lemma}
\begin{proof}[Proof of Lemma \ref{prod:rules}]
In what follows, $\eta \in(0,\frac{1}{2})$ will be fixed, which implies $H_x^{\frac{d}{2}+\eta}\subset L_x^\infty.$
Denote the multivariate binomial coefficients by
\begin{equation*}
	\nchoosek{\alpha}{\alpha'}:=\frac{\prod_{j=1}^{d}\alpha_j!}{\prod_{j=1}^{d}(\alpha_j')!\prod_{j=1}^{d}(\alpha_j-\alpha_j')!}.
\end{equation*}
We begin the proof of \eqref{prod:rule:1} by using the product rule and the triangle inequality:
	\begin{equation}
		\|E^g\cdot\grad_vf\|_{H_{n}^{s}} \leq \sum_{\substack{|\alpha|+|\beta|\leq s \\ \alpha_1\leq \alpha}}\nchoosek{\alpha}{\alpha_1}\left\|\partial_x^{\alpha-\alpha_1}E^g \cdot\grad_v \partial_x^{\alpha_1}\partial_v^\beta f \right\|_{L_{n}^2}.
	\end{equation}
Now, we split into four separate cases:
\begin{description}
	\item[Case 1 $\alpha_1=\alpha$:] In this case, we use H\"older's inequality, the embedding $H^{d/2+\eta}_x \subset L^\infty_x$, and Lemma \ref{interpolation} 
	\begin{align}
		\|E^g\cdot\grad_v\pab f\|_{L_{n}^2} \leq& \|E^g\|_{L^\infty}\|\grad_v\pab f\|_{L_{n}^2} \notag\\
		\leq& C\|g\|_{H_{n_0}^{d/2-1+\eta}}\|\grad_vf\|_{H_{n}^{s}}.\label{prod:rule:1:1}
	\end{align}
In the following cases, $\alpha_1 < \alpha.$
	\item[Case 2 $|\alpha_1|+|\beta| = |\alpha|+|\beta|-1$ or $|\alpha|+|\beta|=1$:] Here, we have $|\alpha-\alpha_1|=1$, and hence similarly to the previous case
	\begin{align}\|\partial_x^{\alpha-\alpha_1}E^g\cdot\grad_v\partial_x^{\alpha_1}\partial_v^\beta f\|_{L_{n}^2} \leq& \|\grad_x E^g\|_{L^\infty} \|\grad_vf\|_{H_{n}^{s}}\notag\\
		\leq&C \|g\|_{H_{n_0}^{\frac{d}{2}+\eta}}\|\grad_vf\|_{H^s_n}. \label{prod:rule:1:2}
	\end{align}
\item[Case 3 $|\alpha-\alpha_1| = |\alpha|+|\beta| \geq 2$:]
Here, we necessarily have $\alpha_1=\beta = 0,$ and hence (using the Sobolev embedding now on $\grad_v f$), 
\begin{align}
	\|\partial_x^\alpha E^g \cdot\grad_v f\|_{L_{n}^2} \leq& \|\partial_x^\alpha E^g\|_{L^2}    \left\| \left\| \grad_v f \right\|_{L_x^\infty} \right\|_{L_{v,m'}^2}
    \notag\\
	\leq& C\|g\|_{H_{n_0}^{s-1}}\|\grad_vf\|_{H_{n}^{d/2+\eta}}.\label{prod:rule:1:3}
\end{align}
\item[Case 4 $|\alpha_1|+|\beta|\leq |\alpha|+|\beta|-2 $:] Here $|\alpha-\alpha_1|\geq2,$ so:
\begin{align}
	\|\partial_x^{\alpha-\alpha_1}E^g\cdot\grad_v\partial_x^{\alpha_1}\partial_v^\beta f\|_{L_{n}^2} \leq& \|\partial_x^{\alpha-\alpha_1}E^g \|_{L^2_x}\left\| \left \|\grad_v\partial_x^{\alpha_1}\partial_v^\beta f \right\|_{L_x^\infty} \right\|_{L_{v,n}^2}
    \notag\\
	\leq& C\|g\|_{H_{n_0}^{|\alpha-\alpha_1|-1}} \|\grad_v\partial_x^{\alpha_1}\partial_v^\beta f\|_{H_{n}^{d/2+\eta}} \notag\\
	\leq& C\|g\|_{H_{n_0}^{s-1}}\|\grad_vf\|_{H_{n}^{s}}.\label{prod:rule:1:4}
\end{align}
\end{description}
Summing over the various cases, \eqref{prod:rule:1} follows.

The proof of \eqref{prod:rule:2} is similar but just slightly more subtle after using the cancellation that occurs when all of the derivatives land on $\grad_v f$.
First, distribute the derivatives with Leibniz's rule. 
\begin{align}
	\sum_{|\alpha|+|\beta|\leq s} \left<\pab(E^g\cdot\grad_vf),\pab f\right>_{n} \leq \sum_{|\alpha|+|\beta|\leq s}\sum_{\alpha_1\leq\alpha} \nchoosek{\alpha}{\alpha_1} \left|\left<\partial_x^{\alpha-\alpha_1}E^g\cdot\grad_v\partial_x^{\alpha_1}\partial_v^\beta f,\pab f\right>_{n}\right|.
\end{align}
We distinguish the following cases:
\begin{description}
	\item[Case 1 $\alpha_1=\alpha $:] by integrating by parts the $\grad_v$ onto the weight, we have 
	\begin{align}
		\left|\left<E^g\cdot\grad_v\pab f,\pab f\right>_{n}\right| =& \frac{1}{2}\iint_{\mathbb T^d \times \mathbb R^d} E^g\cdot\grad_v|\pab f|^2 (\brv^m) \dv\dx\notag\\
		\leq& C\|E^g\|_{L^\infty}\|\pab f\|_{L_{n}^2}^2.
	\end{align}
	\item[Case 2 $|\alpha-\alpha_1|=1$:] Cauchy-Schwarz gives 
	\begin{align}
		\left|\left<\partial_x^{\alpha-\alpha_1}E^g\cdot\grad_v\partial_x^{\alpha_1}\partial_v^\beta f,\pab f\right>_{n}\right| \leq& \|\partial_x^{\alpha-\alpha_1}E^g\|_{L^\infty}\|f\|_{H_{n}^{s}}^2\notag\\
		\leq& C\|\grad_xE^g\|_{L^\infty}\|f\|_{H_{n}^{s}}^2.
	\end{align}
	\item[Case 3 $|\alpha-\alpha_1| \geq 2$:]
We have:  
	\begin{align}
		&\left|\left<\partial_x^{\alpha-\alpha_1}E^g\cdot\grad_v \partial_x^{\alpha_1}\partial_v^\beta f,\pab f\right>_{n}\right| \notag\\
		\leq& \|\partial_x^{\alpha-\alpha_1}E^g\|_{L^{2\frac{s-1}{|\alpha-\alpha_1|-1}}} \left\|\grad_v\partial_x^{\alpha_1}\partial_v^\beta f \right\|_{L_{v,n}^2 L_x^{2\frac{s-1}{|\alpha_1|+|\beta|}}} \|\pab f\|_{L_{n}^2}\notag\\
		\leq& C\|\grad_x E^g\|_{L^\infty}^{\frac{|\alpha_1|+|\beta|}{s-1}}\|\grad_x E^g\|_{H_x^{s}}^{\frac{|\alpha-\alpha_1|-1}{s-1}} \|f\|_{H_{n}^{s}}^2.
	\end{align}
 In the above we used H\"older's inequality, Gagliardo-Nirenberg interpolation \eqref{ineq:iGNS} on $\grad_xE^g$ and Sobolev embedding on the $\grad_v\partial_x^{\alpha_1}\partial_v^\beta f$ term, where we note that the order of integrability $p:=2\frac{s-1}{|\alpha_1|+|\beta|}$ corresponds to the embedding of $H_x^{\tilde\sigma}$ in $ L_x^{p}$ for $\tilde\sigma$ given by:
\begin{equation}
	\tilde\sigma = \frac{d}{2}\frac{|\alpha-\alpha_1|-1}{s-1}
\end{equation}
which satisfies $\tilde\sigma < |\alpha-\alpha_1|-1$ exactly if $s > \frac{d}{2}+1$.
\end{description}
Summing over the above cases, we obtain \eqref{prod:rule:2}, which completes the proof of Lemma \ref{prod:rules}. 
\end{proof}

Next, we prove Lemma \ref{comp:it}. 
\begin{proof}[Proof of Lemma \ref{comp:it}]
We begin with an estimate for $\|f\|_{H_{m'}^{\sigma'}}^2$, which we then use to derive an estimate for $\|f\|_{H_{m'}^{\sigma'}}^p$, for $p> 2.$ Applying It\^o's formula to $\|\pab f^j\|_{L_{m'}^2}^2$, we have:
\begin{align}
	\dee\|\pab f^j\|_{L_{m'}^2}^2 =& -2\left<\pab(v\cdot\grad_x f^j), \pab f^j\right>_{m'}\dt \notag\\
	&+2\left<\Delta_v\pab f^j, \pab f^j\right>_{m'}\dt \notag\\
	&+2\left<\pab(\Div_v(f^jv)),\pab f^j\right>_{m'}\dt \notag\\
	&-2\left<\theta_R\pab(\varphi_\epsilon*E^{j-1}\cdot\grad_vf^j),\pab f^j\right>_{m'}\dt \notag\\
	&-2\left<\pab(\grad_vf^j\cdot\dee W_t),\pab f^j\right>_{m'} \notag\\
	&+\sum_k\left<\pab[(\sev)^2f^j],\pab f^j\right>_{m'}\dt \notag\\
	&+\sum_k \|\pab (\sev f^j)\|_{L_{m'}^2}^2 \dt \notag \\
	=&: \cT_{\alpha,\beta}(f^j) + \cD_{\alpha,\beta}(f^j) + \F_{\alpha,\beta}(f^j) + \cN_{\alpha,\beta}(f^j) + \cM_{\alpha,\beta}(f^j) + \cC_{\alpha,\beta}(f^j), \label{pab:Ito}
 \end{align}
where
\begin{align*}
\cC_{\alpha,\beta}(f^j) = \sum_k\left<\pab[(\sev)^2f^j],\pab f^j\right>_{m'}\dt + \sum_k \|\pab (\sev f^j)\|_{L_{m'}^2}^2 \dt. 
\end{align*}
Here, the $\cT,\cD,\F,\cN,\cM, \cC$ terms abbreviate transport, dissipation, friction, nonlinear electric field, martingale, and correction contributions, respectively.
We begin by observing that by integration by parts, 
\begin{equation}
	\cT_{\alpha,\beta}(f^j)+\F_{\alpha,\beta}(f^j) \leq C \|f^j\|_{H_{m'}^{|\alpha|+|\beta|}}^2\dt.\label{drift:it}
\end{equation}
Similarly, for the dissipative term, integrating by parts gives:
\begin{align}
 	\cD_{\alpha,\beta}(f^j) =& -2\|\grad_v\pab f^j\|_{L_{m'}^2}^2\dt + \iint_{\mathbb T^d \times \mathbb R^d} |\pab f^j|^2 \Delta_v\brv^{m}\dv\dx\dt\notag\\
 	\leq& -2\|\grad_v\pab f^j\|_{L_{m'}^2}^2\dt + C\|f\|_{H_{m'}^{|\alpha|+|\beta|}}^2\dt \label{diss:it}
\end{align}
Next, turn to the It\^o correction terms, which need to be treated carefully in order to not lose derivatives.
Distributing derivatives gives 
	\begin{align}
	\cC_{\alpha,\beta}(f^j) =& \sum_k \left(\|\pab(\sev)f^j\|_{L_{m'}^2}^2 + \left<\pab(\sev)^2f^j,\pab f^j\right>_{m'}\right)\dt \notag\\
	=&\sum_k\sum_{\alpha'<\alpha} \nchoosek{\alpha}{\alpha'} \sigma_k^2\left<\partial_x^{\alpha-\alpha'}(e_k \otimes e_k) : \grad_v^2 \partial_x^{\alpha'}\partial_v^\beta f^j,\pab f^j\right>_{m'}\dt  \label{c:1:it}\\
	&+\sum_k\sigma_k^2\left<(e_k\otimes e_k):\grad_v^2\pab f^j, \pab f^j\right>_{m'}\dt \label{c:2:it}\\
	&+\sum_k\|\pab(\sev)f^j\|_{L_{m'}^2}^2\dt \label{c:3:it}.
\end{align}
Now, \eqref{c:2:it} provides a term of highest order that cancels the It\^o correction $\|\pab(\sev)f^j\|_{L_{m'}^2}^2$, and terms of lower order that can either be readily controlled by $\|f^j\|_{H_{m'}^{\sigma'}}^2$ or cancel out with a corresponding term in \eqref{c:1:it}.
Integrating by parts in \eqref{c:2:it} we have, 
\begin{align}
	&\eqref{c:2:it} = \sigma_k^2\left<(e_k\otimes e_k):\grad_v^2 \pab f^j, \pab f^j\right>_{m'} \notag\\
	=&-\left<\sev\pab f^j,\sev\pab f^j\right>_{m'} \notag\\
	&- \iint_{\mathbb T^d \times \mathbb R^d} \pab f^j (\sev\pab f^j) \sev(\brv^{m'})\dv\dx\notag\\
	=& -\left<\pab(\sev f^j), \sev \pab f^j\right>_{m'} - \left<[\sev,\pab]f^j, \sev \pab f^j\right>_{m'} \notag\\
	&+\frac{1}{2}\iint_{\mathbb T^d \times \mathbb R^d} |\pab f^j|^2 (\sev)^2(\brv^{m'})\dv\dx\notag\\
	=& -\|\sev (\pab f^j)\|_{L_{m'}^2}^2\notag \\
	&- \left<\pab (\sev f^j), [\sev,\pab]f^j\right>_{m'} - \left<[\sev,\pab]f^j,\sev \pab f^j\right>_{m'}\notag\\
	&+\frac{1}{2}\iint_{\mathbb T^d \times \mathbb R^d} |\pab f^j|^2 (\sev)^2(\brv^{m'})\dv\dx\notag\\
	=&-\|\sev(\pab f^j)\|_{L_{m'}^2}^2\label{c:4:it}\\
	&-2\left<[\sev,\pab]f^j, \sev \pab f^j\right>_{m'} \label{c:5:it}\\
	&-\left<[\pab,\sev]f^j,[\sev,\pab]f^j\right>_{m'}\label{c:5:it:extra}\\
	&+\frac{1}{2}\iint_{\mathbb T^d \times \mathbb R^d} |\pab f^j|^2 (\sev)^2(\brv^{m'})\dv\dx. \label{c:6:it}
\end{align}
Thus, we observe that \eqref{c:4:it} cancels the It\^o correction \eqref{c:3:it}, \eqref{c:6:it} is bounded above by $C\|\pab f^j\|_{L_{m'}^2}^2$, and \eqref{c:5:it:extra} only contains derivatives of order lower than $|\alpha|+|\beta|$ and is thus bounded above by $C\|f^j\|_{H_{m'}^{\sigma'}}^2$.
Next, we turn to \eqref{c:5:it}.
This term contains (A) terms from the commutator where the total number of derivatives on $f^j$ is strictly less than $|\alpha|+|\beta|$, which can be treated by integration by parts of the $\sev$ and are thus bounded above by $C\|f^j\|_{H_{m'}^{\sigma'}}^2$; and (B) a highest order term which we deal with as follows: 
\begin{align}
	&2\sum_{\substack{\alpha' < \alpha \\ |\alpha'|=1}}\sigma_k^2\left<\partial_x^{\alpha'}e_k\cdot\grad_v \partial_x^{\alpha-\alpha_1}\partial_v^\beta f^j,e_k\cdot\grad_v\pab f^j\right>_{m'}\notag\\
	=&-2\sum_{\substack{\alpha' < \alpha \\ |\alpha'| =1}} \sigma_k^2\left<(\partial_x^{\alpha'}e_k\cdot\grad_v)(e_k\cdot\grad_v)\partial_x^{\alpha-\alpha'}\partial_v^\beta f^j,\pab f^j\right>_m\notag\\
	&-\sum_{\substack{\alpha'< \alpha\\ |\alpha| = 1}} \sigma_k^2\iint_{\mathbb T^d \times \mathbb R^d} \pab f^j \partial_x^{\alpha'}e_k\cdot\grad_v\partial_x^{\alpha-\alpha'}\partial_v^\beta f^j  e_k\cdot\grad_v(\brv^{m'})\dv\dx\notag\\
	=&-\sum_{\substack{\alpha' < \alpha\\ |\alpha'|=1}} \sigma_k^2\left<\partial_x^{\alpha'}(e_k\otimes e_k):\grad_v^2\partial_x^{\alpha-\alpha'}\partial_v^\beta f^j, \pab f^j\right>_{m'}\label{c:7:it}\\
	&-\sum_{\substack{\alpha'<  \alpha\\ |\alpha| = 1}} \sigma_k^2\iint_{\mathbb T^d \times \mathbb R^d} \pab f^j \partial_x^{\alpha'}e_k\cdot\grad_v\partial_x^{\alpha-\alpha'}\partial_v^\beta f^j  e_k\cdot\grad_v(\brv^{m'})\dv\dx \label{c:8:it}.
\end{align}
Notice that to obtain the prefactor $1$ in \eqref{c:7:it} we used the symmetry of the tensor $e_k\otimes e_k$.
Now, \eqref{c:8:it} can be directly bounded by $C\|f^j\|_{H_{m'}^{\sigma'}}^2$, while \eqref{c:7:it} cancels out the highest order term in \eqref{c:1:it}.
Therefore, we finally conclude using \eqref{c:1:it}--\eqref{c:8:it} that we have 
\begin{equation}
	\cC_{\alpha,\beta}(f^j) \leq C\|f^j\|_{H_{m'}^{\sigma'}}^2\dt .\label{c:it}
\end{equation}
Next we treat the contribution of the electric field term. It follows from \eqref{prod:rule:1:1}-\eqref{prod:rule:1:4} in the proof of \eqref{prod:rule:2} that:
\begin{align}
	\cN_{\alpha,\beta}(f^j) \leq&  C\theta_R(\|f^{j-1}\|_{H_{m_0}^{s_0}}) (\|\varphi_\epsilon*E^{j-1}\|_{W^{1,\infty}}+\|\varphi_\epsilon*E^{j-1}\|_{H_x^{\sigma'-1}})\|f^j\|_{H_{m'}^{\sigma'}}^2\dt\notag\\
	\lesssim_{\epsilon,R}& \|f^j\|_{H_{m'}^{\sigma'}}^2\dt. \label{e:it}
\end{align}
Finally, the martingale contribution is given by
	\begin{align}
	\cM_{\alpha,\beta}(f^j) =& \iint_{\mathbb T^d \times \mathbb R^d} |\pab f^j|^2 \dee W_t \cdot\grad_v(\brv^{m'}) \dv \dx - 2\left<[\pab,\dee W_t]\grad_vf^j,\pab f^j\right>_{m'}. \label{m:it}
\end{align}
We sum \eqref{drift:it}-\eqref{m:it} over $|\alpha|+|\beta|\leq \sigma'$ and obtain 
\begin{align}
	\dee \|f^j\|_{H_{m'}^{\sigma'}}^2 \leq& C\|f^j\|_{H_{m'}^{\sigma'}}^2\dt -2\nu\|\grad_vf^j\|_{H_{m'}^{\sigma'}}^2\dt +\sum_{0\leq |\alpha| + |\beta| \leq \sigma} \cM_{\alpha,\beta}(f^j),
\end{align}
so integrating in time and using the Burkh\"older--Davis--Gundy inequality (see e.g. \cite{da1996ergodicity}) (hereinafter abbreviated as BDG) we obtain:
\begin{align}
	\E\sup_{t\leq T}\|f^j(t)\|_{H_{m'}^{\sigma'}}^2 \leq& \E\|f_0\|_{H_{m'}^{\sigma'}}^2 + C\int_0^T\E\|f^j(t)\|_{H_{m'}^{\sigma'}}^2\dt \notag\\
	&+ C\E\left(\int_0^T\|f^j(t)\|_{H_{m'}^{\sigma'}}^4\dt\right)^{\frac{1}{2}} \notag\\
	\leq& \|f_0\|_{H_{m'}^{\sigma'}}^2 + C\int_0^T\E\|f^j(t)\|_{H_{m'}^{\sigma'}}^2\dt \notag\\
	&+ \frac{1}{2}\E\sup_{t\leq T}\|f^j(t)\|_{H_{m'}^{\sigma'}}^2,
\end{align}
where the second line followed from H\"older's inequality.
After rearranging and applying Gr\"onwall's inequality we obtain the uniform-in-$j$ estimate:
\begin{equation}
	\E\sup_{t'\leq T}\|f^j(t)\|_{H_{m'}^{\sigma'}}^2 \leq C\E\|f_0\|_{H_{m'}^{\sigma'}}^2,
\end{equation}
where the constant $C$ depends on $\epsilon,R,T,m',\sigma'$, but not $f_0$ or $j$. Thus we have obtained \eqref{unif:bdd:it} for $p=2$.

Now, we use It\^o's formula again, this time for $\|f^j\|_{H_{m'}^{\sigma'}}^p$, with $p>2$: 
\begin{align}
	\dee \|f^j\|_{H_{m'}^{\sigma'}}^p =& \frac{p}{2}\|f^j\|_{H_{m'}^{\sigma'}}^{p-2}\dee\|f\|_{H_{m'}^{\sigma'}}^2 \notag\\
	&+ \frac{p}{2}\frac{p-2}{2} \|f^j\|_{H_{m'}^{\sigma'}}^{p-4} \sum_{k}\sum_{|\alpha|+|\beta|\leq \sigma'} \left|\left<\pab(\sev f^j),\pab f^j\right>_{m'}\right|^2\dt. \notag
\end{align}
The latter term is treated by a straightforward commutator estimate, and together with the above estimates on $\dee \norm{f^j}_{H^{\sigma'}_{m'}}^2$, we obtain 
\begin{align}
	\dee \|f^j\|_{H_{m'}^{\sigma'}}^p \leq& C\|f^j\|_{H_{m'}^{\sigma'}}^p\dt -p\|f^j\|_{H_{m'}^{\sigma'}}^{p-2}\sum_{|\alpha|+|\beta|\leq\sigma'}\left<\pab (\grad_vf^j\cdot\dee W_t),\pab f^j\right>_{m'}.\notag
\end{align}
After integrating in time, using the BDG inequality, and applying H\"older's inequality, we have 
\begin{align}
\E\sup_{t\leq T}\|f^j\|_{H_{m'}^{\sigma'}}^p \leq& \E\|f_0^j\|_{H_{m'}^{\sigma'}}^p + C\int_0^T\E\|f^j\|_{H_{m'}^{\sigma'}}^p\ds\notag\\
&+C\E\left(\int_0^T\|f^j\|_{H_{m'}^{\sigma'}}^{2p}\ds\right)^{\frac{1}{2}}\notag\\
\leq& \E\|f_0\|_{H_{m'}^{\sigma'}}^p + C\int_0^T\E\|f^j\|_{H_{m'}^{\sigma'}}^p\ds + \frac{1}{2}\E\sup_{t'\leq T}\|f^j\|_{H_{m'}^{\sigma'}}^p. 
\end{align}
By rearranging and using Gr\"onwall's lemma, we obtain \eqref{unif:bdd:it}.

We now turn to the proof of \eqref{equic:it}.
We have:
\begin{align*}
	\E\left\|f^j\right\|_{W^{\alpha,p}([0,T];H_{m'-1}^{\sigma'-2})}^p \leq& C \|f_0\|_{H_{m'-1}^{\sigma'-2}}^p\\
	&+C\E\left\|\int_0^t\left(-v\cdot\grad_xf^j + \Delta_v f^j + \Div_v(f^jv)\right)\ds\right\|_{W^{1,p}([0,T];H_{m'-1}^{\sigma'-2})}^p\\
	&+C\E\left\| \int_0^t \theta_R\varphi_\epsilon*E^{j-1}\cdot\grad_vf^j \right\|_{W^{1,p}([0,T];H_{m'-1}^{\sigma'-2})}^p\\
	&+C\E\left\| \frac{1}{2}\int_0^t\sum_k(\sev)^2f^j \ds \right\|_{W^{1,p} ([0,T] ; H_{m'-1}^{\sigma'-2} )}^p\\
	&+C\E\left\| \int_0^t \grad_vf^j \cdot\dee W_t \right\|_{W^{\alpha,p}([0,T];H_{m'-1}^{\sigma'-2})}^p.
\end{align*}
The terms that are regular in time are estimated in a straightforward manner using the available regularity: 
\begin{align*}
	\E\left\|\int_0^t \left(v\cdot\grad_xf^j + \Div_v(f^jv)\right)\ds\right\|_{W^{1,p}([0,T]; H_{m'-1}^{\sigma'-2})}^p \leq C\E\sup_{t'\leq T}\|f^j(t')\|_{H_{m'}^{\sigma'-1}}^p\\
	\E\left\|\int_0^t \Delta_v f^j \ds \right\|_{W^{1,p}([0,T]; H_{m'-1}^{\sigma'-2})}^p \leq C\E\sup_{t'\leq T} \|f^j(t')\|_{H_{m'-1}^{\sigma'}}^p\\
	\E\left\|\int_0^t \sum_k (\sev)^2f^j\right\|_{W^{1,p}([0,T];H_{m'-1}^{\sigma'-2})}^p \leq C\E\sup_{t'\leq T}\|f^j(t')\|_{H_{m'-1}^{\sigma'}}^p\\
	\E\left\|\int_0^t \theta_R \varphi_\epsilon*E^{j-1}\cdot\grad_vf^j\ds\right\|_{W^{1,p}([0,T];H_{m'-1}^{\sigma'-2})}^p \leq C_R \E\sup_{t'\leq T}\|f^j(t')\|_{H_{m'-1}^{\sigma'-1}}^p. 
    \end{align*}
The time-regularity is only limited by the stochastic integral, which is estimated by a variant of the BDG inequality adapted to fractional regularity estimates in time (see e.g.  [Lemma 2.1; \cite{flandoli1995martingale}] for a proof), namely
\begin{align*}
\E\left\| \int_0^t \grad_vf^j \cdot \dee W_t \right\|_{W^{\alpha,p}([0,T];H_{m'-1}^{\sigma'-2})}^p & \leq  C\E\int_0^T \norm{\grad_vf^j(s)}_{H^{\sigma'-2}_{m'-1}}^p ds \\
 & \leq C \E\sup_{t'\leq T}\|f^j(t')\|_{H_{m'-1}^{\sigma'-1}}^p. 
\end{align*}
Therefore, using that $W^{1,p}([0,T];H_{m'-2}^{\sigma'-2}) \subset W^{\alpha,p}([0,T];H_{m'-2}^{\sigma'-2})$ continuously and \eqref{unif:bdd:it}, we obtain:
\begin{equation}
	\E\|f^j\|_{W^{\alpha,p}([0,T];H_{m'-2}^{\sigma'-2})}^p \leq C_{R,T}\E\|f_0\|_{H_{m'}^{\sigma'}}^p
\end{equation}
uniformly in $j$, which implies \eqref{equic:it}, completing the proof of Lemma \ref{comp:it}.  
\end{proof}

\begin{remark}
By examining the proof above, one can see that one can also treat magnetic fields, due to the special structure of the Lorentz force $v \times B(x)$, which ensures both $\grad_v \cdot (v \times B) = 0$ and, despite the power of $v$, the estimates do not lose any moments in $v$ as $v \times B$ is orthogonal to $v$ (nor does the $v$ dependence create any issues controlling higher regularity).  
\end{remark}

We continue the proof of Lemma \eqref{glbl:reg:reg}. 
The approximation procedure mixes $f^j$ and $f^{j+1}$ in a way that makes it difficult to apply the usual method of using tightness of the laws in pathspace and applying the Skorohod embedding theorem to construct probabilistically weak solutions which are subsequently upgraded to strong solutions (see e.g. \cites{debussche2011local,debussche2012global,GV14,brzezniak2020well}).
Instead we will prove that $\set{f^j}_{j=1}^\infty$ is Cauchy in a suitable topology.
For this we first need the following consequence of Lemma \ref{comp:it} and the Borel-Cantelli lemma. 
\begin{lemma} \label{lem:BC} 
For all $\delta > 0$, $\exists$ a $\mathcal{F}_1$-measurable, almost-surely finite, random constant $C_0$ such that for \emph{all} $j \geq 0$ there holds 
\begin{align*}
\sup_{s < 1} \norm{f^j(s)}_{H^{\sigma'}_{m'}} < C_0(\omega,\delta) \brak{j}^{\delta}.  
\end{align*}
Moreover, $\forall \alpha,n$ there holds, 
\begin{align*}
\PP(C_0 > n) \lesssim_{\delta,\alpha} n^{-\alpha}. 
\end{align*}
\end{lemma}
\begin{proof}
Recall the uniform in $j$ bound \eqref{unif:bdd:it} for the iterates for $T=1$:
\begin{equation*}
	\sup_{j\geq 1}\E\sup_{s\leq 1}\|f^j(s)\|_{H_{m'}^{\sigma'}}^p \leq C_{p,R,\epsilon,M} <\infty.
\end{equation*}
This estimate implies:
\begin{equation*}
	\EE\sum_{j=0}^\infty\frac{\sup_{s\leq 1}\|f^j(s)\|_{H_{m'}^{\sigma'}}^p}{\brak{j}^{\delta p}} \leq C_{p,R,\epsilon, M},
\end{equation*}
for $p>\frac{1}{\delta}$. Denote by $A_j$ the sets:
\begin{equation*}
	A_j := \{\omega \in \Omega: \, \sup_{s\leq 1}\|f^j(s)\|_{H_{m'}^{\sigma'}}>\brak{j}^\delta\},
\end{equation*}
and note that by Chebyshev's inequality:
\begin{equation}
	\sum_{j=0}^\infty \P(A_j) \leq \sum_{j=0}^\infty \frac{\E\sup_{s\leq 1} \|f^j(s)\|_{H_{m'}^{\sigma'}}^p}{\brak{j}^{\delta p}} < \infty. 
\end{equation}
It then follows by the Borel--Cantelli lemma that
\begin{equation*}
	\P(\limsup_{j\to \infty}A_j) = 0,
\end{equation*}
implying that $\P$-a.s., $\sup_{s\leq 1}\|f^j(s)\|_{H_{m'}^{\sigma'}}>\brak{j}^\delta$ at most for a finite number of $j$'s. Denote the largest such $j$ by $j_0(\omega)$. We then see that there is a random constant $C_0(\omega,\delta)$ such that \begin{equation*}\sup_{j\geq 0}\left(\brak{j}^{-\delta}\sup_{s\leq 1}\|f^j(s)\|_{H_{m'}^{\sigma'}}\right) < C_0(\omega,\delta)
\end{equation*}
$\P$--almost surely. In particular, we can take:
\begin{equation*}
	C_0(\omega,\delta) := \inf\left\{n \in \mathbb{N}: \, \sup_{j\leq j_0(\omega)} \left(\brak{j}^{-\delta}\sup_{s\leq 1}\|f^j(s)\|_{H_{m'}^{\sigma'}}\right) < n \right\}.
\end{equation*}
To bound the probability that $C_0$ is large, we observe:
\begin{align*}
	\P(C_0 > n ) \leq& \P\left(\sup_{j\geq 0} \left(\brak{j}^{-\delta}\sup_{s\leq 1}\|f^j(s)\|_{H_{m'}^{\sigma'}}>n\right)\right)\\
	\leq& \sum_{j=0}^\infty \brak{j}^{-\delta p}\E\sup_{s\leq1}\|f^j(s)\|_{H_{m'}^{\sigma'}}^pn^{-p}\\
	\lesssim& n^{-p}.
\end{align*}
This completes the proof of the lemma.
\end{proof}
The next lemma is the crucial convergence estimate. 
\begin{lemma}\label{iter:Cauchy}
There exists an increasing sequence of stopping times $\tau_n$ such that $\set{f^j}_{j=1}^\infty$ is Cauchy in $L^2_\omega C([0,\tau_n];H^{s_0}_{m_0})$
and the stopping time 
\begin{align*}
\lim_{n \to \infty} \tau_n = \xi, 
\end{align*}
is almost-surely greater than $1$. 
\end{lemma}
\begin{proof}
Define the increasing sequence of stopping times
\begin{align*}
\tau_n = \inf \set{ t: \exists j: \norm{f^j(t)}_{H^{s_0+1}_{m_0}} > n \brak{j}^\delta}. 
\end{align*}
Note that by Lemma \ref{lem:BC} there holds 
\begin{align*}
\PP(\tau_n \geq 1) & = \PP\left( \sup_{j \geq 0} \brak{j}^{-\delta} \sup_{t < 1} \norm{f^j(t)}_{H^{s_0+1}_{m_0}} < n \right) \\
& \geq \PP(C_0 < n)\\
& \geq 1 - \PP(C_0 > n) \\
& \geq 1 - \mathcal{O}(n^{-\alpha}). 
\end{align*}
Therefore, $\lim_{n \to \infty} \PP(\tau_n > 1) =1$ and so if we define
\begin{align*}
\xi = \lim_{n \to \infty} \tau_n, 
\end{align*}
then $\xi$ is almost-surely greater than or equal to $1$.  

Let $\delta \in (0,1/6)$ be fixed arbitrary. 
We will show by induction that $\exists K_0 > 0$ (deterministic constant depending on $\delta$) such that for all $j \geq 1$, there holds 
\begin{align}
\EE \sup_{s<t \wedge \tau_n} \norm{f^{j} - f^{j-1}}_{H^{s_0}_{m_0}}^2 & \leq \frac{(K_0 n^4 t)^j j^{4\delta j}}{j!}. \label{ineq:PowerSer}
\end{align}
First consider the case $j=1$.
The calculation of $\dee \norm{\pab (f^1 - f^0)}^2_{L^2_{m_0}}$ is the same in Lemma \ref{comp:it} except for the nonlinear terms. 
That is, for $|\alpha|+|\beta|\leq s_0$ we have for some constant $C > 0$
\begin{align*}
\dee\|\pab (f^1 - f^0)\|_{L_{m_0}^2}^2 & \leq C\norm{f^1 - f^0}_{H^{s_0}_{m_0}}^2 \dt \\  
    & \quad -2\left<\theta_R(\norm{f^0}_{H^{s_0}_{m_0}})\pab(\varphi_\eps \ast E^{0}\cdot\grad_vf^1),\pab (f^1-f^0)\right>_{m_0}\dt \notag\\
	& \quad -2\left<\pab(\grad_v(f^1-f^0)\cdot\dee W_t),\pab (f^1-f^0)\right>_{m_0}. 
\end{align*}
For the nonlinear term we note that by \eqref{prod:rule:2} we have, recalling the definition of $\tau_n$
\begin{align*}
\abs{\left<\theta_R(\norm{f^0}_{H^{s_0}_{m_0}})\pab(\varphi_\eps \ast E^{0}\cdot\grad_vf^1),\pab (f^1-f^0)\right>_{m_0}} & \lesssim_{\eps,R} n\norm{f^1 - f^0}_{H^{s_0}_{m_0}}.
\end{align*}
Integrating in time and using the BDG inequality as above, we obtain (note that $f^1$ and $f^0$ have the same initial data), 
\begin{align*}
\E\sup_{s < t \wedge \tau_n}\|f^1(s) - f^0(s)\|_{H_{m_0}^{s_0}}^2 \leq& C\int_0^{t \wedge \tau_n}\E\|f^1(s) - f^0(s)\|_{H_{m_0}^{s_0}}^2\ds \notag\\
 & \quad + C n^2 t + C\E\left(\int_0^{t \wedge \tau_n}\|f^1(s) - f^0(s)\|_{H_{m_0}^{s_0}}^4\ds \right)^{\frac{1}{2}} \notag\\
	\leq& C n^2 t + C\int_0^{t \wedge \tau_n} \E\|f^1(s) - f^0(s)\|_{H_{m_0}^{s_0}}^2\ds \notag\\
	&+ \frac{1}{2}\E\sup_{s\leq t \wedge \tau_n}\|f^1(s) -f^0(s)\|_{H_{m_0}^{s_0}}^2,
\end{align*}
Therefore, Gr\"onwall's inequality verifies \eqref{ineq:PowerSer} for $j =1$ and some large $K_0$.

Next consider the inductive step.
Hence, suppose that \eqref{ineq:PowerSer} holds for $j$ and we wish to verify that it holds for $j+1$. 
As above, for some constant $C > 0$
\begin{align*}
\dee\|\pab (f^{j+1} - f^j)\|_{L_{m_0}^2}^2 & \leq C\norm{f^{j+1} - f^j}_{H^{s_0}_{m_0}}^2 \dt \\  
& \quad -2\left<\theta_R(\norm{f^j}_{H^{s_0}_{m_0}})\pab(\varphi_\eps \ast E^{j}\cdot\grad_vf^{j+1}),\pab(f^{j+1}-f^j)\right>_{m_0}\dt \\
&\quad+ 2\left<\theta_R(\norm{f^{j-1}}_{H^{s_0}_{m_0}})\pab(\varphi_\eps \ast E^{j-1}\cdot\grad_vf^j),\pab (f^{j+1}- f^j) \right>_{m_0}\dt \notag\\
	& \quad -2\left<\pab(\grad_v(f^{j+1}-f^j)\cdot\dee W_t),\pab (f^{j+1}-f^j)\right>_{m_0}. 
\end{align*}
The nonlinearity separates into several natural terms, namely
\begin{align*}
(\ast) & = -2\left<\theta_R(\norm{f^j}_{H^{s_0}_{m_0}})\pab(\varphi_\eps \ast E^{j}\cdot(\grad_vf^{j+1} - \grad_v f^j)),\pab (f^{j+1}- f^j) \right>_{m_0}\dt \notag\\
& \quad  -2\left<\left(\theta_R(\norm{f^j}_{H^{s_0}_{m_0}}) - \theta_R(\norm{f^{j-1}}_{H^{s_0}_{m_0}}) \right) \pab(\varphi_\eps \ast E^{j}\cdot \grad_vf^{j}), \pab (f^{j+1}- f^j) \right>_{m_0}\dt \notag \\
& \quad -2\left<\theta_R(\norm{f^{j-1}}_{H^{s_0}_{m_0}})\pab( (\varphi_\eps \ast E^{j} -  \varphi_\eps \ast E^{j-1}) \cdot \grad_v f^j),\pab (f^{j+1}- f^j) \right>_{m_0}\dt \notag\\
& = \mathcal{N}_1  + \mathcal{N}_2  + \mathcal{N}_3. 
\end{align*}
The term $\mathcal{N}_1$ is treated via \eqref{prod:rule:2} in the same manner as in Lemma \ref{comp:it}, giving
\begin{align*}
\mathcal{N}_1 \lesssim_{R,\eps} \norm{f^{j+1} - f^j}_{H^{s_0}_{m_0}}^2. 
\end{align*}
The terms $\mathcal{N}_2,\mathcal{N}_3$ however are different.
The term $\mathcal{N}_3$ is estimated via the following for $t < \tau_n$:
\begin{align*}
\mathcal{N}_3 & \lesssim_R \norm{f^j}_{H^{s_0+1}_{m_0}}\norm{f^j - f^{j-1}}_{H^{s_0}_{m_0}} \norm{f^{j+1} - f^j}_{H^{s_0}_{m_0}}\\
& \lesssim  n j^{\delta} \norm{f^j - f^{j-1}}_{H^{s_0}_{m_0}} \norm{f^{j+1} - f^j}_{H^{s_0}_{m_0}}.
\end{align*}
The term $\mathcal{N}_2$ requires a control on the difference $\theta_R(\norm{f^j}_{H_{m_0}^{s_0}})-\theta_R(\norm{f^{j-1}}_{H_{m_0}^{s_0}})$:
\begin{align*}
\theta_R(\norm{f^j}_{H^{s_0}_{m_0}}) &- \theta_R(\norm{f^{j-1}}_{H^{s_0}_{m_0}}) \\
&= \int_0^1 \theta_R'( z \norm{f^j}_{H^{s_0}_{m_0}} + (1-z)\norm{f^{j-1}}_{H^{s_0}_{m_0}}) (\norm{f^j}_{H^{s_0}_{m_0}} - \norm{f^{j-1}}_{H^{s_0}_{m_0}}) \dee z. 
\end{align*}
Therefore, for $t < \tau_n$
\begin{align*}
\mathcal{N}_2 & \lesssim \abs{\norm{f^j}_{H^{s_0}_{m_0}} - \norm{f^{j-1}}_{H^{s_0}_{m_0}}} \norm{f^j}_{L^2_{m_0}} \norm{f^j}_{H^{s_0+1}_{m_0}} \norm{f^{j+1} - f^j}_{H^{s_0}_{m_0}} \\
& \lesssim \norm{f^j- f^{j-1}}_{H^{s_0}_{m_0}} \norm{f^j}_{L^2_{m_0}} \norm{f^j}_{H^{s_0+1}_{m_0}} \norm{f^{j+1} - f^j}_{H^{s_0}_{m_0}} \\
& \lesssim n^2 j^{2\delta} \norm{f^j- f^{j-1}}_{H^{s_0}_{m_0}}\norm{f^{j+1} - f^j}_{H^{s_0 }_{m_0}}. 
\end{align*}
Integrating in time and using the BDG inequality as above, we obtain (noting that $f^{j+1}$ and $f^j$ have the same initial data) for $t < \tau_n$:
\begin{align*}
\E\sup_{s < t }\|f^{j+1}(s) - f^j(s)\|_{H_{m_0}^{s_0}}^2 \leq
& C\int_0^{t}\E\|f^{j+1}(s) - f^j(s)\|_{H_{m_0}^{s_0}}^2\ds \notag\\
& \quad + C_0 n^4 j^{4\delta} \EE \int_0^t \norm{f^j(s) - f^{j-1}(s)}_{H^{s_0}_{m_0}}^2 \dee s \\ 
& \quad + \frac{1}{2}\E\sup_{s\leq t}\|f^{j+1}(s) -f^{j}(s)\|_{H_{m_0}^{s_0}}^2 \\
\end{align*} 
By the inductive hypothesis
\begin{align*}
C n^4 j^{4\delta} \EE \int_0^t \norm{f^j(s) - f^{j-1}(s)}_{H^{s_0}_{m_0}}^2 \dee s
& \leq C_0 n^4 j^{4\delta} \int_0^t \frac{(C_0 n^4 s)^j j^{4\delta j}}{j!} \dee s \\
& \leq \frac{(C_0 n^4 t)^{j+1} j^{4\delta (j+1)}}{(j+1)!}, 
\end{align*}
and so we have verfied \eqref{ineq:PowerSer}. 

Finally, we show that \eqref{ineq:PowerSer} implies that $\set{f^j}$ is Cauchy in $L_\omega^2L_t^2([0,\tau_n];H_{m_0}^{s_0})$.
Indeed, let $k < \ell$ and 
\begin{align}
\EE \sup_{s<t \wedge \tau_n} \norm{f^{\ell} - f^{k}}_{H^{s_0}_{m_0}}^2 & \leq \sum_{j = k}^\ell \frac{(K_0 n^4 t)^j j^{4\delta j}}{j!}. 
\end{align}
Hence, if we choose $k > (2C_0 n^4 t)^{1/\delta}$, then
\begin{align*}
\EE \sup_{s<t \wedge \tau_n} \norm{f^{\ell} - f^{k}}_{H^{s_0}_{m_0}}^2 & \leq \sum_{j = k}^\ell \frac{1}{2^{-\delta j}} \frac{k^{j\delta} j^{4\delta j}}{j!}  \leq \sum_{j = k}^\ell \frac{1}{2^{-\delta j}} \frac{j^{5\delta j}}{j!}. 
\end{align*}
By Stirling's formula we have the following uniformly in $j$ (using $5\delta < 1$), 
\begin{align*}
\frac{j^{5\delta j}}{j!} \lesssim_\delta 1,
\end{align*}
therefore
\begin{align*}
\EE \sup_{s<t \wedge \tau_n} \norm{f^{\ell} - f^{k}}_{H^{s_0}_{m_0}}^2 & \lesssim \sum_{j = k}^\ell \frac{1}{2^{-\delta j}} \lesssim \frac{1}{2^{-\delta k}}.  
\end{align*}
We conclude that the sequence is Cauchy as claimed in the lemma. 
\end{proof} 

\begin{lemma}\label{iterates:converge}
For each $n,$ the iterates $\set{f^j}_{j=1}^\infty$ converge uniformly in $H_{m_0}^{s_0}$ on compact subintervals of $[0,\tau_n]$ to a strong pathwise solution of the SPDE \eqref{SVPFP:reg:reg} on the set $\set{\tau_n > 0} \subset \Omega$. 
\end{lemma} 
\begin{proof}
Consider only $\omega \in \set{\tau_n > 0} \subset \Omega$. 
Let $f$ be the limiting process of the $f^j$ in $L_\omega^2C([0,\tau_n];H_{m_0}^{s_0})$ - whose existence is guaranteed by Lemma \ref{iter:Cauchy}. We will show that each term in \eqref{iteration} converges to the corresponding term in \eqref{SVPFP:reg:reg}.
The convergence of the linear terms is straightforward:
\begin{align}
	\E \sup_{t\leq T\wedge \tau_n}\left\|\int_0^tv\cdot\grad_x(f^{j+1}-f)\ds\right\|_{H_{m_0-1}^{s_0-1}}^2 & \lesssim_{m_0} T^2 \E\sup_{t\leq T\wedge\tau_n}\|f^{j+1}-f\|_{H_{m_0}^{s_0}}^2 \to 0,\\
	\E\sup_{t\leq T\wedge \tau_n}\left\|\int_0^t \Delta_v(f^{j+1}-f)\ds\right\|_{H_{m_0}^{s_0-2}}^2 & \lesssim T^2 \E\sup_{t\leq T\wedge\tau_n}\|f^{j+1}-f\|_{H_{m_0}^{s_0}}^2 \to 0,\\
	\E\sup_{t\leq T\wedge\tau_n}\left\|\int_0^t \Div_v(f^{j+1}v-fv)\ds\right\|_{H_{m_0-1}^{s_0-1}}^2 & \lesssim_{m_0} T^2\E\sup_{t\leq T\wedge\tau_n}\|f^{j+1}-f\|_{H_{m_0}^{s_0}}^2 \to 0, \\ 
	\E\sup_{t\leq T\wedge\tau_n}\left\|\int_0^t\sum_k(\sev)^2(f^{j+1}-f)\ds\right\|_{H_{m_0}^{s_0-2}}^2 & \lesssim T^2\E\sup_{t\leq T\wedge\tau_n}\|f^{j+1}-f\|_{H_{m_0}^{s_0}}^2 \to 0.
\end{align}
For the nonlinear electric field terms, we have:
\begin{equation*}
	\E\sup_{t\leq T\wedge\tau_n} \left\|\int_0^t( \theta_{R}^j\varphi_\epsilon*E^j\cdot\grad_vf^j - \theta_{R}\varphi_\epsilon*E\cdot\grad_vf)\ds\right\|_{H_{m_0}^{s_0-1}} \lesssim \mathcal{N}_1+\mathcal{N}_2+\mathcal{N}_3,
\end{equation*}
where:
\begin{gather*}
	\mathcal{N}_1 := \E\sup_{t\leq T\wedge \tau_n}\left\|\int_0^t (\theta_R^j-\theta_R)\varphi_\epsilon*E^j\cdot\grad_vf^j\ds\right\|_{H_{m_0}^{s_0-1}},\\
	\mathcal{N}_2 := \E\sup_{t\leq T\wedge\tau_n}\left\|\int_0^t \theta_R\varphi\varphi_\epsilon*(E^j-E)\cdot\grad_vf^j \ds\right\|_{H_{m_0}^{s_0-1}},\\
	\mathcal{N}_3 := \E\sup_{t\leq T\wedge\tau_n}\left\|\int_0^t\theta_R\varphi_\epsilon*E\cdot\grad_v(f^j-f)\ds\right\|_{H_{m_0}^{s_0-1}}.
\end{gather*}
These terms are estimated as follows:
\begin{align*}
	\cN_1 \lesssim_{R}& \E\sup_{t\leq T\wedge\tau_n}\int_0^t\left|\|f^j\|_{H_{m_0}^{s_0}}-\|f\|_{H_{m_0}^{s_0}}\right|\|f^j\|_{H_{m_0}^{s_0}}^2\ds\\
	\lesssim&\E\sup_{t\leq T\wedge\tau_n}\left(\|f^j-f\|_{L^2([0,t];H_{m_0}^{s_0})}\|f^j\|_{L^4([0,T];H_{m_0}^{s_0})}^2\right)\notag\\
	\to& 0,
\end{align*}
\begin{align*}
		\cN_{2} \lesssim \E\sup_{t\leq T\wedge\tau_n}\left( \left\|f^j-f\right\|_{L^2([0,t];H_{m_0}^{s_0})}\|f^j\|_{L^2([0,t];H_{m_0}^{s_0})}\right) \to 0,
\end{align*}
\begin{align*}
	\cN_{3} \lesssim \E\sup_{t\leq T\wedge\tau_n} \left(\|f\|_{L^2([0,t]; H_{m_0}^{s_0})}\|f^j-f\|_{H_{m_0}^{s_0}}\right) \to 0.
\end{align*}

Lastly, for the martingale terms we use the BDG inequality:
\begin{align}
	\E\sup_{t\leq T\wedge\tau_n}\left\|\int_0^t \sum_k\sev(f^{j+1}-f)\cdot\dee W_s^k\right\|_{H_{m_0}^{s_0-1}}^2 \lesssim& \E\int_0^T\|f^{j+1}-f\|_{H_{m_0}^{s_0}}^2\ds\notag\\
	\lesssim& T^2\E\sup_{t\leq T\wedge\tau_n}\|f^{j+1}-f\|_{H_{m_0}^{s_0}}^2\notag\\
	\to& 0.
\end{align}
Combining the above, we see that $f$ is a solution of \eqref{SVPFP:reg:reg}.
\end{proof}
\begin{corollary}
There exists a global, strong pathwise solution of the SPDE \eqref{SVPFP:reg:reg} such that $\forall p \in [2,\infty$, $f \in  L_{\omega}^{p}C_{t,loc}H_{m'-2}^{\sigma'-3}\cap L_\omega^{p}L_{t,loc}^\infty H_{m'}^{\sigma'}$. 
\end{corollary}
\begin{proof}
By sending $n \to \infty$ and using that $\tau_n$ is a non-decreasing sequence such that $\lim_{n \to \infty} \mathbb P( \tau_n >1) = 1$ we see that almost-surely,
$\set{f^j}_{j=1}^\infty$ converges uniformly in  $H_{m_0}^{s_0}$  on compact subintervals of $[0,1)$ to a limiting function $f \in C_t([0,1);H_{m_0}^{s_0})$. 
By Sobolev interpolation, and uniform boundedness in $L^\infty_{t,\loc} H^{\sigma'}_{m'}$, we obtain similar uniform convergence in $H_{m''}^{s''}$ for all $0 \leq s'' < s'$ and $m'' < m'$. 
At the same time, the uniform bounds on $\set{f^j}$ from Lemma \ref{comp:it} imply that $\forall p \in [2,\infty)$, $f \in  L_{\omega}^{p}C_tH_{m'-2}^{\sigma'-3}\cap L_\omega^{p}L_{t,loc}^\infty H_{m'}^{\sigma'}$  by the lower semicontinuity of weak convergence. 
By Lemma \ref{iterates:converge}, the limiting function $f$ is also a solution of \eqref{SVPFP:reg:reg}.
Now, we simply iterate the construction starting at $t = 1/2,3/2,...$ to obtain the existence of a global solution satisfying the desired bounds. 
\end{proof}

The following lemma proves uniqueness of solutions to \eqref{SVPFP:reg:reg}, thus completing the proof of Lemma \ref{glbl:reg:reg}. 

\begin{lemma}\label{uniqueness:reg:reg}
Let $f, f'$ be two global  pathwise solutions to \eqref{SVPFP:reg:reg} on the same stochastic basis with $f(0)=f'(0)=f_0$ for some $\F_0$-measurable $f_0$ with $\E\|f_0\|_{H_{m'}^{\sigma'}}^p < \infty$ for some $p>2$ and such that for all $\eps > 0$, $ f, f' \in L_{\omega}^{p-\eps}C_tH_{m'-2}^{\sigma'-3}\cap L_\omega^{p-}L_{t,loc}^\infty H_{m'}^{\sigma'}.$
Then $f, f'$ are indistinguishable, that is:
	\begin{equation}
		\P\left(f(t)=f'(t) \, \text{ for all } 0 \leq t  \right)=1. \label{r:r:indist}
	\end{equation}
\end{lemma}
\begin{proof}
This is proved by an energy estimate on $\|f-f'\|_{H_{m'-1}^{\sigma'-1}}^2.$ Similarly to the proof of Lemma \ref{comp:it}, for $|\alpha| + |\beta| \leq \sigma'-1$ we have:
\begin{align}
\dee\|\pab(f-f')\|_{L_{m'-1}^2}^2 \leq& C\|f-f\|_{H_{m'-1}^{\sigma'-1}}^2\dt\notag\\
&-2\brak{\pab(\theta_R\varphi_\epsilon*E\cdot\grad_vf - \theta_{R}'\varphi_\epsilon*E'\cdot\grad_vf'),\pab(f-f')}_{m'-1}\dt\notag\\
&-2\int_0^t\brak{\pab(\grad_v(f-f')\cdot\dee W_t),\pab(f-f')}_{m'-1}\label{reg:reg:en:est}.
\end{align}
We split the electric field contributions as:
\begin{align*}
	\brak{\pab(\theta_R\varphi_\epsilon*E\cdot\grad_vf-\theta_R'\varphi_\epsilon*E'\cdot\grad_vf'),\pab(f-f')}_{m'-1} = \cN_1 +\cN_2 + \cN_3,
\end{align*}
where:
\begin{gather*}
	\cN_1 := \brak{(\theta_R-\theta_R')\pab(\varphi_\epsilon* E\cdot\grad_vf),\pab(f-f')}_{m'-1}, \\
	\cN_2 := \brak{\theta_R'\pab(\varphi_\epsilon*(E-E')\cdot\grad_vf),\pab (f-f')}_{m'-1}, \\
	\cN_3 := \brak{\theta_R'\pab(\varphi_\epsilon*E'\cdot\grad_v(f-f')),\pab(f-f')}_{m'-1}.
\end{gather*}
These are estimated as follows:
\begin{equation}
	|\cN_1| \lesssim_R \|f\|_{H_{m'}^{\sigma'}}^2\|f-f'\|_{H_{m'-1}^{\sigma'-1}}^2,\label{reg:reg:n1}
\end{equation}
\begin{equation}
	|\cN_2| \lesssim \|f\|_{H_{m'}^{\sigma'}}\|f-f\|_{H_{m'-1}^{\sigma'-1}}^2,\label{reg:reg:n2}
\end{equation}
\begin{equation}
	|\cN_3| \lesssim_R \|f-f'\|_{H_{m'-1}^{\sigma'-1}}^2, \label{reg:reg:n3}
\end{equation}
where in \eqref{reg:reg:n1} we used the mean value theorem for $\theta_R$ and \eqref{prod:rule:1} , in \eqref{reg:reg:n2} we used \eqref{prod:rule:1}, and in \eqref{reg:reg:n3} we used \eqref{prod:rule:2} - in addition to Lemma \ref{interpolation} for each electric field. 

Now, fix $K>0$. Since $f,f' \in L_\omega^{p-}L_{t,loc}^\infty H_{m'}^{\sigma'}$, the stopping time:
\begin{equation*}
	\xi_K = \inf\{t\geq 0: \sup_{s\leq t}\|f(s)\|_{H_{m'}^{\sigma'}} + \sup_{s\leq t}\|f'(s)\|_{H_{m'}^{\sigma'}}>K\}\wedge\tau\wedge\tau'
\end{equation*}
is almost surely finite. Even though it is not clear that $\xi_K$ is almost surely positive in general, for almost every $\omega\in\Omega$ there exists $K>0$ such that $\xi_K>0,$ and in addition $\xi_K\to \tau\wedge\tau'$ $\P$--a.s. as $K\to \infty$. With this in mind, we fix $T>0$ and use \eqref{reg:reg:n1}-\eqref{reg:reg:n3} and the BDG inequality in \eqref{reg:reg:en:est}, to obtain:
\begin{align}
	\E\sup_{s\leq t\wedge\xi_K}\|f(s)-f'(s)\|_{H_{m'-1}^{\sigma'-1}}^2 \lesssim& \E\int_0^t\sup_{s\leq s'\wedge \xi_K}\|f(s)-f'(s)\|_{H_{m'-1}^{\sigma'-1}}^2\ds' \notag\\
	&+\E\left(\int_0^t\sup_{s\leq s'\wedge\xi_K}\|f(s)-f'(s)\|_{H_{m'-1}^{\sigma'-1}}^4\ds'\right)^{\frac{1}{2}}\notag\\
	\leq& C\E\int_0^t\sup_{s\leq s'\wedge\xi_K}\|f(s)-f'(s)\|_{H_{m'-1}^{\sigma'-1}}^2\ds'\notag\\
	&+\frac{1}{2}\E\sup_{s\leq t\wedge\xi_K}\|f(s)\|_{H_{m'-1}^{\sigma'-1}}^2,
\end{align}
for all $t\leq T$, whereby the usual rearrangement and Gr\"onwall's lemma give:
\begin{equation*}
	\E\sup_{s\leq T\wedge\xi_K}\|f(s)-f'(s)\|_{H_{m'-1}^{\sigma'-1}}^2 = 0.
\end{equation*}
Taking $K\to \infty$ and then $T\to \infty$, the conclusion follows.
\end{proof}

\subsection{Proof of Lemma \ref{glbl:reg}}

Next, we want to pass to the limit $\eps \to 0$, for which we need uniform-in-$\eps$ estimates similar to those of
Lemma \ref{comp:it}, but this time for a family $\set{f_\eps}_{\eps >0}$ of solutions to \eqref{SVPFP:reg:reg}. Note that since $f_\epsilon \in L_{t,loc}^\infty H_{m'}^{\sigma'}\cap C_tH_{m'-2}^{\sigma'-3},$ the highest norm in which we know $f_\epsilon$ is continuous is $C_tH_{m'-1}^{\sigma'-1}$ - and thus we use this as the base for our estimates.
\begin{lemma}\label{comp:epsilon}
	Let $f$ be a solution of \eqref{SVPFP:reg:reg}. For $\alpha\in(0,\frac{1}{2}),$ $p\geq 2,$ we have the uniform in $\epsilon$ estimates:
	\begin{equation}
		\E\sup_{t\leq T} \|f(t)\|_{H_{m'-1}^{\sigma'-1}}^p \lesssim_{p,T,R,f_0} 1 \label{unif:lim}
	\end{equation}
and
\begin{equation}
	\E\|f\|_{W^{\alpha,p}([0,T];H_{m'-2}^{\sigma'-3})}^p \lesssim_{p,T,R,f_0} 1.\label{equic:lim}
\end{equation}
\end{lemma}
\begin{proof}
The proof proceeds by induction in the number of derivatives\footnote{see for instance \cite{luk2016strichartz} for similar inductive energy estimates for the relativistic Vlasov--Maxwell system.}.
The inductive hypothesis is that for $s > d/2$ derivatives on a solution $f$ of \eqref{SVPFP:reg:reg}, we have:
\begin{equation}
	\E\sup_{t'\leq T}\|f\|_{H_{m'-1}^s}^p \lesssim_{p,R,T,f_0} 1.
\end{equation}
We show that this implies the same estimate for $s+1$. Begin by using It\^o's formula on $\|\pab f\|_{L_{m'-1}^2}^2$ for $|\alpha|+|\beta|=s+1,$ where similarly to \eqref{pab:Ito} we obtain:
\begin{align}
	\dee\|\pab f\|_{L_{m'-1}^2}^2 =& -2\left<\pab(v\cdot\grad_x f), \pab f\right>_{m'-1}\dt \notag\\
	&+2\left<\Delta_v\pab f, \pab f\right>_{m'-1}\dt \notag\\
	&+2\left<\pab(\Div_v(fv)),\pab f\right>_{m'-1}\dt \notag\\
	&-2\left<\theta_R(\|f\|_{H_{m_0}^{s_0}})\pab(\varphi_\epsilon*E\cdot\grad_vf),\pab f\right>_{m'-1}\dt \notag\\
	&-2\left<\pab(\grad_vf\cdot\dee W_t),\pab f\right>_{m'-1} \notag\\
	&+\sum_k\left<\pab[(\sev)^2f],\pab f\right>_{m'-1}\dt \notag\\
	&+\sum_k \|\pab (\sev f)\|_{L_{m'-1}^2}^2 \notag \\
	=& \cT_{\alpha,\beta}(f) + \cD_{\alpha,\beta}(f) + \F_{\alpha,\beta}(f) + \cN_{\alpha,\beta}(f) + \cM_{\alpha,\beta}(f) + \cC_{\alpha,\beta}(f).\label{unif:lim:enest}
\end{align}
The linear terms are treated as in the proof of Lemma \ref{comp:it}, and the only term that requires new attention is $\cN_{\alpha,\beta}(f).$ By the classical Gagliardo-Nirenberg inequality (see e.g. [Proposition A.3 \cite{tao2006nonlinear}]  we have:
\begin{align}
	&\left|\left<\pab (\varphi_\epsilon* E\cdot\grad_vf),\pab f\right>_{m'-1}\right|\notag\\
	 \leq& C\|\varphi_\epsilon*E\|_{W^{1,\infty}}\|f\|_{H_{m'-1}^{s+1}}^2\dt\notag\\
	&+\sum_{\substack{\gamma<\alpha\\|\alpha-\gamma|\geq 2}}\|\partial_x^{\alpha-\gamma}\varphi_\epsilon* E\|_{L_x^{2\frac{s}{|\alpha-\gamma|-1}}}
	\|\grad_v\partial_v^\beta\partial_x^\gamma f\|_{L_{v,m}^2 L_x^{2\frac{s}{|\beta|+|\gamma|+1}}} \|\pab f\|_{L_{m'-1}^2}\dt\notag\\
	\leq& C\|\varphi_\epsilon* E\|_{W^{1,\infty}}\|f\|_{H_{m'-1}^{s+1}}^2\dt\notag\\
	&+C\sum_{\substack{\gamma<\alpha\\|\alpha-\gamma|\geq 2}}\|\grad_x\varphi_\epsilon*E\|_{L_x^\infty}^{\frac{|\beta|+|\gamma|+1}{s}}\|\varphi_\epsilon*E\|_{H_x^{s+1}}^{\frac{|\alpha-\gamma|-1}{s}} \|\grad_v\partial_v^\beta f\|_{L_{m'-1}^2}^{1-\eta}\|f\|_{L_{m'-1}^2}^{1+\eta}\dt,\label{unif:lim:ef:interm}
\end{align}
where for fixed $\gamma,$ the interpolation parameter $\eta$ is given by:
\begin{align}
	\eta =& \frac{|\gamma|}{|\alpha|-1} + \frac{d}{|\alpha|-1}\left(\frac{1}{2}-\frac{|\beta|+|\gamma|+1}{2s}\right)\notag\\
	=&\frac{|\gamma|}{|\alpha|-1} +\frac{d}{|\alpha|-1}\cdot\frac{|\alpha-\gamma|-1}{2s}
\end{align}
and thus $\eta<1$ provided $s > \frac{d}{2}$. By Young's inequality and \eqref{unif:lim:ef:interm} it follows that
\begin{equation}
	\cN_{\alpha,\beta} \leq C\left(1+\|f\|_{H_{m'-1}^{s}}^p\right)\dt + CR\|f\|_{H_{m'-1}^{s+1}}^2\dt. \end{equation}
Plugging this back into \eqref{unif:lim:enest}, and using the same procedure as in the proof of Lemma \ref{comp:it}, we obtain:
\begin{align}
	\dee\|f\|_{H_{m'-1}^{s+1}}^2 \leq& C(1+R)\|f\|_{H_{m'-1}^{s+1}}^2\dt + \sum_{|\alpha|+|\beta| = s+1} \cM_{\alpha,\beta}(f) + \cM_{0,0}(f)\notag\\
	&+C(1+\|f\|_{H_{m'-1}^s}^p)\dt.
\end{align}
We again integrate in time and apply the BDG inequality as in the proof of \eqref{unif:bdd:it} for $p=2$, where the only difference is the term $\|f\|_{H_{m'-1}^s}^p,$ which is now controlled by the inductive hypothesis, and we get:
\begin{equation}
	\E\sup_{t\leq T} \|f\|_{H_{m'-1}^{s+1}}^2\lesssim_{R,p,T,f_0} 1. 
\end{equation}
With this in hand, we can directly transfer the proof of \eqref{unif:bdd:it} for $p>2$ and obtain \eqref{unif:lim}. Then the same argument as the proof of \eqref{equic:it} (i.e. using the variant of BDG from \cite{flandoli1995martingale}*{Lemma 2.1}) gives \eqref{equic:lim}. 

The last thing that remains is to demonstrate the inductive base of the preceding scheme.
Here this is done by  first estimating the $H_{m'}^{2}$ norm of $f$. 
This is sufficient to start the inductive scheme above in $1 \leq d \leq 3$ as $2 > d/2$. 
As the linear terms are always controlled in the same way, we only focus on the electric field contributions. As always, we have:
\begin{align}
	\left|\left<\pab(E\cdot\grad_vf),\pab f\right>_{m'}\right| \leq C\|E\|_{W^{1,\infty}}\|f\|_{H_{m'}^2}^2 
	+ \sum_{\substack{\gamma<\alpha\\|\alpha-\gamma|\ge 2}}\left<\partial_x^{\alpha-\gamma}E\cdot\grad_v\partial_v^\beta\partial_x^\gamma f,\pab f\right>_{m'},
\end{align}
but since only two derivatives are acting on $f$ at this point, the terms in the summation are only present when $|\alpha|=2$ and $|\beta|=|\gamma|=0.$
Let $q = 4$ in $d=1,2$, and for $d \geq 3$ let $q$ be arbitrary such that $2 < q < \frac{2d}{d-2}$.
Then by H\"older's inequality and Sobolev embeddings we have  
\begin{align}
	\left|\left<\partial_x^\alpha E \cdot\grad_v f, \partial_x^\alpha f\right>_{m'}\right| \leq& \|\grad_x^2E\|_{L_x^{\frac{q}{q-1}}} \|\grad_vf\|_{L_{v,m}^2L_x^q}\|f\|_{H_{m'}^2}\notag\\
	\lesssim& \|\grad_x^2E\|_{H_x^{\frac{d(q-2)}{2q}}}\|f\|_{H_{m'}^2}^2\notag\\
	\lesssim& \|f\|_{H_{m_0}^{s_0}}\|f\|_{H_{m'}^2}^2, \label{ef:two:der}
\end{align}
where we have used that $s_0 > \frac{d}{2}+1 > \frac{d(q-2)}{2q}$ and the embedding $H_x^1 \subset L_x^q$ which holds for all $d \geq 1$ due to our choice of $q$.
From this point on the procedure is the same as in the inductive step. We plug \eqref{ef:two:der} into \eqref{unif:lim:enest} for $\alpha,\beta$ with $|\alpha|+|\beta| = 2$, sum over all such $\alpha,\beta$ as well as the case when $\alpha=\beta=0,$ integrate in time, apply the BDG inequality and Gr\"onwall's lemma and obtain: 
\begin{equation}
	\E\sup_{t\leq T} \|f(t)\|_{H_{m'}^2}^2 \lesssim_{T,R,f_0} 1.
\end{equation}
Then applying the same argument as in the proof of \eqref{unif:bdd:it} for $p>2,$ we also obtain for $p>2$:
\begin{equation}
	\E\sup_{t\leq T} \|f(t)\|_{H_{m'}^2}^p \lesssim_{p,T,R,f_0} 1.
\end{equation}
This provides the inductive base and therefore the proof of the lemma is complete for $1 \leq d \leq 3$.
\end{proof}

For solutions to \eqref{SVPFP:reg:reg}, it is unclear how to prove $\set{f^\eps}_{\eps > 0}$ forms a Cauchy sequence as $\eps \to 0$.
Instead, we employ a standard procedure based on the Skorokhod embedding theorem (see e.g. \cite{ikeda2014stochastic}) to produce probabilistically weak (called \emph{martingale}) solutions on a new stochastic basis, and then upgrade them to probabilistically strong using the Gy\"ongy--Krylov lemma from \cite{gyongy1996existence} (see Lemma \ref{GK} below).  
We let $(\eps_n)_{n=1}^\infty$ be a decreasing sequence of positive numbers with $\eps_n\to 0$ as $n\to \infty$ and define the corresponding sequence $f_n := f^{\eps_n}$ of  solutions to \eqref{SVPFP:reg:reg}, which we have shown satisfy the uniform bounds \eqref{unif:lim} and \eqref{equic:lim}. 
For $\alpha \in (0,\frac{1}{2})$ and $p>2$ such that $\alpha p>1$, we define the pathspace
\begin{equation}
	\Xbf := W_{loc}^{\alpha,p}([0,\infty); H_{m'-2}^{\sigma'-3}) \cap L_{loc}^\infty([0,\infty); H_{m'-1}^{\sigma'-1}).
\end{equation}
Recall that since $\alpha p >1$ and $H_{m'-2}^{\sigma'-3}\subset H_{m'-3}^{\sigma'-4}$ compactly, from \cite{flandoli1995martingale}*{Theorem 2.2}, we have:
\begin{equation*}
	W_{loc}^{\alpha,p}([0,\infty);H_{m'-2}^{\sigma'-3}) \subset C_tH_{m'-3}^{\sigma'-4}.
\end{equation*}
By the uniform estimates \eqref{unif:lim} and \eqref{equic:lim}, the laws $\nu^n :=\cL(f_n)$ are bounded in probability in $\Xbf$, and thus they are \textit{tight} in the smaller pathspace
\begin{equation}
	\Xbf_{c} := C([0,\infty);H_{m'-3}^{\sigma'-4})\cap L_{loc}^\infty([0,\infty);H_{m'-1}^{\sigma'-1}).
\end{equation}
Note that the tightness in $L_{loc}^\infty([0,\infty);H_{m'-1}^{\sigma'-1})$ is in the weak-$\star$ topology. We now use this to obtain a martingale solution to \eqref{SVPFP:reg} in high regularity.

\begin{proposition}\label{martingale:reg}
	Let $\mu_0$ be a probability measure on $H_{m'}^{\sigma'}$ so that $\int_{H_{m'}^{\sigma'}}\|f\|_{H_{m'}^{\sigma'}}^p\dee\mu_0(f)$ for some $p>2$. Then there exists a stochastic basis $\tilde S = (\tilde\Omega,\tilde\F,\{\tilde\F_t\},\tilde P)$ and a predictable process
	\begin{equation}
		\tilde f \in L_\omega^{p-}C_tH_{m'-3}^{\sigma'-4}\cap L_\omega^{p-}L_{t,loc}^\infty H_{m'-1}^{\sigma'-1}
	\end{equation}
such that $\cL(\tilde f(0))= \mu_0$ and $\tilde f$ solves \eqref{SVPFP:reg} in the sense that, there is a sequence of i.i.d Brownian motions $\set{\widetilde{W}_t^{(k)}}$ such that the following equality holds in $C([0,\infty);H^{\sigma'-3}_{m'-2})$ 
\begin{align*}
	\tilde{f}(t) = f_0 + \int_0^t\left(-v\cdot\nabla_x\tilde{f}(s)-\theta_{R}(\norm{f}_{H^{s_0}_{m_0}}) \tilde{E}(s)\cdot\nabla_v\tilde{f}(s)+\nu\cL \tilde{f}(s)\right)\ds-\int_0^t\nabla_v \tilde{f}(s)\circ\dee \widetilde{W}_s 
	\, \P\text{--a.s.}, 
\end{align*}
with
\begin{align*}
\tilde{E} = \grad_x (-\Delta_x)^{-1} \left( \int_{\mathbb R^d} \tilde{f}(t,\cdot,v) \dee v - 1 \right). 
\end{align*}
\end{proposition}
\begin{proof}
Let $\mu^n = \cL(f_n,\mathcal{W})$ in $\Xbf_c\times C([0,\infty);\mathfrak{U}_0).$ The sequence $(\mu^n)_{n=1}^\infty$ is tight by the uniform estimates \eqref{unif:lim} and \eqref{equic:lim} combined with the fact that its projection onto $C([0,\infty);\mathfrak{U}_0)$ is the same for each $n$. By Prokhorov's theorem, $(\mu^n)_{n=1}^\infty$ has a weakly convergent subsequence - reindexed to $\mu^n$. By Skorokhod's embedding theorem, there exists a new probability space $(\tilde \Omega, \tilde \F, \tilde \P)$ and on it random elements $(\tilde f_n, \tilde{\mathcal{W}}_n) $ with laws $\mu^n$ which converge $\tilde\P$--a.s. to some limit $(\tilde f, \tilde{\mathcal{W}})$ in the product topology of $\Xbf_c\times C([0,\infty);\mathfrak{U}_0)$. Then by a variation of the mollification technique employed in the proof of \cite{bensoussan1995stochastic}*{Equation 4.17} \footnote{See also \cite{brzezniak2020well}*{Proposition 3.2, (iii)} for an application to the primitive equations where the noise is present as a stochastic transport term.} the random elements $(\tilde f_n, \tilde{\mathcal{W}}_n)$ satisfy \eqref{SVPFP:reg:reg} just like $(f_n,\mathcal{W})$, but in the new probability space $(\tilde \Omega,\tilde \F, \tilde \P)$:

\begin{equation*}
	\tilde f_n(t)-\tilde f_n(0) + \int_0^t\left( v\cdot\grad_x\tilde f_n+\tilde E_n\cdot\grad_v\tilde f_n - \frac{1}{2}\sum_k(\sev)^2\tilde f_n - \nu\cL\tilde f_n\right)\ds +\int_0^t\grad_v\tilde f_n\cdot\dee \tilde W_t^n = 0,
\end{equation*}
where we have denoted by $\tilde W_t^n$ the external electric field corresponding to $\tilde{\mathcal{W}}_n$; that is to say, if:
\begin{equation*}
	\tilde{\mathcal{W}}_n = \sum_k g_k \tilde W_k^n,
\end{equation*}
then $\tilde W^n$ is simply given by:
\begin{equation*}
	\tilde W^n = \sum_k \sigma_k e_k \tilde W_k^n.
\end{equation*}

The passage to the limit $n\to \infty$ in the SPDEs \eqref{SVPFP:reg:reg} satisfied by $(\tilde f_n,\tilde{\mathcal{W}}_n)$ to obtain that the limit $(\tilde f,\tilde{\mathcal{W}})$ solves \eqref{SVPFP:reg} can be carried out by combining the convergences $\tilde f_n \to \tilde f$ and $\tilde{\mathcal{W}_n} \to \tilde{\mathcal{W}}$ with \cite{debussche2011local}*{Lemma 2.1}, so we omit the proof for technicalities. We simply note that the presence of the transport noise does not cause any additional difficulties in our setting.
\end{proof}
As our goal is to construct solutions in $\mathcal{S}$, we need to \say{upgrade} the martingale solutions of the preceding lemma to probabilistically strong solutions.
With the above in mind, we now state the Gy\"ongy--Krylov lemma from \cite{gyongy1996existence}, which will allow us to combine the tightness of $(f_n)_{n=1}^\infty$ with the pathwise uniqueness of the limit (Lemma \ref{uniqueness:reg} below) to show that in fact \eqref{SVPFP:reg} has (unique) global solutions on the original stochastic basis $\mathcal{S}$.
\begin{lemma}[Gy\"ongy--Krylov]\label{GK}
	Let $(Y_n)_{n=1}^\infty$ be a sequence of $X$-valued random variables, where $X$ is a complete separable metric space. Then $(Y_n)_{n}$ converges in probability if and only if for every two subsequences $Y_{n_k},Y_{l_k}$ the joint sequence $(Y_{n_k},Y_{l_k})$ has a subsequence $(Y_{n_{k'}},Y_{l_{k'}})$ whose laws converge weakly to a probability measure $\nu$ supported on the diagonal of $X\times X$:
	\begin{equation}
		\nu\left(\left\{(x,y)\in X\times X: \, x=y \right\}\right) = 1. \label{GK:diag}
	\end{equation}
\end{lemma}
With this lemma at hand, we now set to prove pathwise uniqueness of solutions to \eqref{SVPFP:reg}, which is the content of the following:
\begin{lemma}\label{uniqueness:reg}
	Let $f, f'$ be global solutions to \eqref{SVPFP:reg} on the same stochastic basis with $f(0) = f'(0) = f_0$ almost surely, where $\E\|f_0\|_{H_{m'}^{\sigma'}}^p<\infty$ for some $p>2$ and such that $f,f' \in L_\omega^{p-}C_tH_{m'-3}^{\sigma'-4}\cap L_\omega^{p-}L_{t,loc}^\infty H_{m'-1}^{\sigma'-1}$. Then $f,f'$ are indistinguishable, that is:
	\begin{equation}
		\P\left(f(t)=f'(t) \text{ for all } t\geq 0\right) = 1. \label{r:indist}
	\end{equation}
\end{lemma}
\begin{proof}
	First of all, notice that since $f,f'\in L_{t,loc}^\infty H_{m'-1}^{\sigma'-1},$ for $K>0$ the stopping times
	\begin{equation}\xi_K := \inf\{t>0: \|f\|_{H_{m'-1}^{\sigma'-1}}^2 + \|f'\|_{H_{m'-1}^{\sigma'-1}}^2 \geq K\}
	\end{equation} are well defined and satisfy $\xi_K\to \infty$ as $K\to \infty$, $\P$--almost surely.
	We now perform an energy estimate on $H_{m_0}^{s_0}$. We use It\^o's formula on the quantity $\|\pab(f-f')\|_{L_{m_0}^2}^2,$ for $|\alpha|+|\beta| = s_0$:
	\begin{align}
		\dee\|\pab(f-f')\|_{L_{m'}^2}^2 =& -2\left<\pab(v\cdot\grad_x (f-f')),\pab(f-f')\right>_{m_0}\dt\notag\\
		&+2\left<\Delta_v\pab(f-f'), \pab(f-f')\right>_{m_0}\dt\notag\\
		&+2\left<\pab\Div_v((f-f') v), (f-f')\right>_{m_0}\dt\notag\\
		&-2\left<\pab(\theta_RE^f\cdot\grad_vf-\theta_R'E^{f'}\cdot\grad_vf'), \pab(f-f')\right>_{m_0}\dt\notag\\
		&-2\left<\pab(\grad_v(f-f')\cdot\dee W_t),\pab(f-f')\right>_{m_0}\notag\\
		&+\sum_k \left<\pab((\sev)^2(f-f')),\pab(f-f')\right>_{m_0}\dt\notag\\
		&-\sum_k \|\pab(\sev (f-f'))\|_{L_{m_0}^2}^2\dt.
	\end{align}
Clearly, all terms except those involving the electric fields can be estimated as in the proof of \eqref{unif:bdd:it}, so we only examine the electric field term:
\begin{align}
	&\left|\left<\pab(\theta_R E^f\cdot\grad_vf - \theta_R' E^{f'}\cdot\grad_vf'),\pab(f-f')\right>_{m_0}\right| \notag\\
\leq& \left|(\theta_R-\theta_R')\left<\pab(E^f\cdot\grad_vf), \pab(f-f') \right>_{m_0}\right|\notag\\
&+\left|\theta_{R'}\left<\pab((E^f-E^{f'})\cdot\grad_vf'),\pab(f-f')\right>_{m_0}\right|\notag\\
&+\left|\theta_{R'}\left<\pab(E^{f'}\cdot\grad_v(f-f')),\pab (f-f')\right>_{m_0}\right|\notag\\
\leq& C_R\left|\|f\|_{H_{m_0}^{s_0}}-\|f'\|_{H_{m_0}^{s_0}}\right|\cdot\|E^f\|_{H_x^{s_0}}\|f-f'\|_{H_{m_0}^{s_0}}\notag\\
&+C\|E^f-E^{f'}\|_{H^{s_0}}\|\grad_vf\|_{H_{m_0}^{s_0}}\|f-f'\|_{H_{m_0}^{s_0}}\notag\\
&+CR\|f-f\|_{H_{m_0}^{s_0}}^2\notag\\
\lesssim_{R,K}&\|f-f'\|_{H_{m_0}^{s_0}}^2
\end{align}
for $t \leq \xi_K$, where we have used the mean value theorem on $\theta_R$ and Lemma \ref{prod:rules}. Thus, arguing similarly to the proof of Lemma \ref{uniqueness:reg:reg} (i.e. by the BDG and Gr\"onwall inequalities), we have:
\begin{equation}\E\sup_{t' \leq t\wedge \xi_K\wedge T}\|f-f'\|_{H_{m_0}^{s_0}}^2 \leq C_{R,K,T}\|f(0)-f'(0)\|_{H_{m_0}^{s_0}}^2 = 0.
\end{equation}
Since $\xi_K\to \infty$ as $K\to \infty,$ the monotone convergence theorem implies that for all $T>0$ we have:
\begin{equation}
	\E\sup_{t\leq T} \|f(t)-f'(t)\|_{H_{m_0}^{s_0}}^2 = 0,
\end{equation}
which implies \eqref{r:indist} since $T$ is arbitrary.
\end{proof}
We now have everything we need to (subsequentially) pass to the limit $\epsilon \to 0$ in the original stochastic basis.
\begin{proof}[Proof of Lemma \ref{glbl:reg}]
We define the joint laws $\mu^{n,l}=\cL( f_n, f_l, \mathcal{W}),$ Similarly to the discussion in Lemma \ref{martingale:reg}, for any sequence $\mu^{n_k,l_k}$ with $n_k,l_k\to \infty$ as $k\to \infty,$ by Prokhorov's theorem the estimates \eqref{unif:lim} and \eqref{equic:lim} (and the fact that $\{\mathcal{W}\}$ is a singleton) provide a weakly convergent subsequence of probability measures in $\Xbf_c\times\Xbf_c\times C([0,\infty);\mathfrak{U}_0)$, which we still denote (after relabelling) by $\mu^{n_k,l_k}$, and we denote its limit by $\mu$. By the Skorokhod embedding theorem, we can construct a new stochastic basis again denoted by $\tilde{\mathcal{S}}=(\tilde{\Omega},\tilde{F},\tilde{\P})$ and on it a sequence of random elements $(\tilde f_{n_k},\tilde f_{l_k},\tilde{\mathcal{W}}^k)$ and $(\tilde f, \tilde{\tilde{f}}, \tilde{\mathcal{W}})$ such that $\cL(\tilde f_{n_k},\tilde f_{l_k},\tilde{\mathcal{W}}^k) = \mu^{n_k,l_k}$, $\cL(\tilde f, \tilde{\tilde f},\tilde{\mathcal{W}})$ and:
\begin{gather*}
	(\tilde f_{n_k},\tilde f_{l_k},\tilde{\mathcal{W}^k}) \to (\tilde f, \tilde{\tilde f},\tilde{\mathcal{W}})\quad \text{ in } \Xbf_c\times \Xbf_c\times C([0,\infty);\mathfrak{U}_0), \quad \tilde\P\text{--a.s..}
\end{gather*} 
As in Lemma \ref{martingale:reg}, $(\tilde f_{n_k},\tilde{\mathcal{W}}^k)$ and $(\tilde f_{l_k},\tilde{\mathcal{W}}^k)$ satisfy the SPDE \eqref{SVPFP:reg:reg} in the new stochastic basis (by the method of \cite{bensoussan1995stochastic}*{Section 4.3.4}), so we can pass to the limit $k\to \infty$ in all the terms of \eqref{SVPFP:reg:reg} for $(\tilde f_{n_k},\tilde{\mathcal{W}}^k)$ and $(\tilde f_{l_k},\tilde{\mathcal{W}^k})$ (using \cite{debussche2011local}*{Lemma 2.1} for the stochastic integrals) to show that $(\tilde f,\tilde{\mathcal{W}})$ $(\tilde{\tilde{f}},\tilde{\mathcal{W}})$ are solutions to \eqref{SVPFP:reg} on the new stochastic basis. Since $\P(f_{n_k}(0)=f_{l_k}(0))=1$, we also have $\tilde\P(\tilde f_{n_k}(0)=\tilde f_{l_k}(0))=1$, and thus in the limit $k\to\infty$ we obtain $\tilde\P(\tilde f(0) = \tilde{\tilde f}(0))=1$. Therefore, by Lemma \ref{uniqueness:reg}, $\tilde f$ and $\tilde{\tilde f}$ are indistinguishable. This means that the measure $\bar{\mu}$, defined as the projection of $\mu$ onto the first two components $\Xbf_c\times\Xbf_c$, is in fact supported on the diagonal of $\Xbf_c\times\Xbf_c$. Thus, by Lemma \ref{GK}, a subsequence of $(f_n)$ converges in probability in the original stochastic basis $\mathcal{S}$ in the topology of $\Xbf_c'$ to a limiting process $f$ which solves \eqref{SVPFP:reg}. This concludes the proof of Lemma \ref{glbl:reg}.
\end{proof}


\section{Proof of main theorem} \label{sec:iv}
In this section, with the results of Section \ref{sec:iii} at hand, we prove the main result of the paper, Theorem \ref{thm:main}. At first, we consider initial data $f_0$ satisfying $\|f_0\|_{H_m^\sigma}\leq M<\infty$ almost surely, for a fixed deterministic $M$. This assumption can be removed at the end by a cutting argument similar to that of Lemma \ref{lem:locHiRegSVPFP}. We treat the initial data $f_0$ with a sequence of regularization and velocity cutoff operators $\mathcal{R}^n,$ obtaining a sequence of regularized data 
\begin{equation*}
	f_0^n := \mathcal{R}^n f_0, 
\end{equation*}
defined as
\begin{equation*}
	\mathcal{R}^n f := \theta_n(v)n^{2d}\eta\left(\frac{\cdot}{n}\right)\ast_{x,v}f, 
\end{equation*}
where $\eta \in C_c^\infty(\mathbb{R}^{2d})$ satisfies $\eta(x,v)\geq 0$ and $\iint \eta(x,v)\dv\dx = 1$.
We note the following properties of this regularization. The proofs are standard and are omitted for brevity.
\begin{lemma} \label{lem:RegProp}
Let $s' \geq s \geq 0$ be integers and $m' \geq m \geq 0$ be arbitrary.
Then,
\begin{itemize}
\item[(i)] The regularization operators $\mathcal{R}^n$ are uniformly bounded on $H^{\sigma}_m$
\begin{align*}
\sup_{n \geq 1}\norm{\mathcal{R}^n f}_{H^{\sigma}_{m}} \lesssim_{m,\sigma} \norm{f}_{H^{\sigma}_m}. 
\end{align*}
\item[(ii)] The regularization operators satisfy: for $f \in H^{\sigma}_m$
\begin{align*}
\norm{\mathcal{R}^n f}_{H^{\sigma'}_{m'}} \lesssim n^{\sigma'-\sigma} n^{(m' - m)/2} \norm{f}_{H^{\sigma}_m}. 
\end{align*}
\item[(iii)] The regularization operators converge in the following senses: for $f \in H^{\sigma}_m$ there holds 
\begin{align}
& \lim_{n \to \infty} \norm{\mathcal{R}^n f - f}_{H^{\sigma}_m} = 0 \\
& \lim_{n \to \infty} n\norm{\mathcal{R}^n f - f}_{H^{\sigma-1}_m} = 0. \label{ineq:RnCopt}
\end{align}
\end{itemize}
\end{lemma}

In the previous section, we showed that each of the $f_0^n$ generates a maximal solution $f^n$ of \eqref{SVPFP} in $C([0,\tau_n);H_m^\sigma)\cap L_{\text{loc}}^\infty([0,\tau_n);H_{m'-1}^{\sigma'-1})$, where $\tau_n$ is the maximal time of existence of $f^n$. We now show that the sequence $(f^n)_{n=1}^\infty$ of approximate solutions has a strongly convergent subsequence.

We start by defining the stopping times:
\begin{gather}
	\tau_n^T := \inf\left\{t\geq 0:\, \|f^n(t)\|_{H_m^\sigma} > \|f_0^n\|_{H_{m}^\sigma}+2\right\}\wedge T,\\
	\tau_{n,l}^T = \tau_n^T \wedge\tau_k^T
\end{gather}
The following is similar to \cite{mikulevicius2004stochastic}*{Lemma 37} or \cite{GV14}*{Lemma 7.1}:
\begin{lemma} \label{lemma:ACL}
Let $\tau_n$ be a sequence of stopping times and suppose that a sequence of predictable processes $f^n \in C([0,\tau_n];H_m^\sigma)$ satisfy:
\begin{equation} 
\lim_{n\to\infty} \sup_{l\geq n}\E\sup_{t'\leq \tau_{n,l}^T}\|f^n-f^l\|_{H_{m}^{\sigma}}^2 = 0, \label{ACL:1}
\end{equation}
	\begin{equation}
	\lim_{\epsilon\to 0}\sup_{n\geq 1}\P\left[\sup_{t\leq \tau_n^T \wedge \epsilon}\|f^n\|_{H_m^\sigma} > \|f_0^n\|_{H_m^\sigma}+1\right]=0.\label{ACL:2}
\end{equation}
Then, there exists a stopping time $\tau$ with $\P(0<\tau\leq T)=1,$ a predictable process $f\in C([0,\tau];H_m^\sigma),$ and a subsequence $(f^{n_j})_{j=1}^\infty$ of $(f^n)_{n=1}^\infty$ such that
\begin{equation}
	\sup_{t\leq \tau}\|f^{n_j}(t)-f(t)\|_{H_m^\sigma} \to 0 \text{ as } j\to \infty,\text{ } \P\text{--a.s.}
\end{equation}
and 
\begin{equation}
	\sup_{t\leq \tau}\|f(t)\|_{H_m^\sigma} \leq 2 + \sup_n\|f_0^n\|_{H_m^\sigma},\text{ } \P\text{--a.s.}
\end{equation}
\end{lemma}
We now verify that the regularized solutions $\set{f^n}_{n \geq 1}$ satisfy the conditions of Lemma \ref{lemma:ACL}.
\begin{lemma}\label{ACL:ver}
	The solutions $(f^n)_{n=1}^\infty$ generated by the regularized data $(f_0^n)_{n=1}^\infty$ satisfy \eqref{ACL:1} and \eqref{ACL:2}.
\end{lemma}
\begin{proof}
	We begin with proving that the Cauchy property \eqref{ACL:1} holds for the sequence $f^n$ of solutions to \eqref{SVPFP:reg} with initial data $f_0^n$.
This is done via an energy estimate with some similarities with the uniqueness and convergence proofs in Section \ref{sec:iii} with one significant difference. 
For $t\leq \tau_{n,l}^T$, we have:
	\begin{align}
		\dee \|\pab(f^n - f^l)\|_{L_m^2}^2 =& -2\left<\pab(v\cdot\nabla_x(f^n-f^l)),\pab (f^n-f^l)\right>_m\dt\notag\\
		&+2\left<\Delta_v(\pab(f^n-f^l)),\pab(f^n-f^l)\right>_m\dt\notag\\
		&+2\left<\pab\Div_v((f^n-f^l)v),\pab(f^n-f^l)\right>_m\dt\notag\\
		&-2\left<\pab( E^n \cdot\nabla_v f^n - E^l\cdot\nabla_vf^l), \pab (f^n-f^l)\right>_m\dt\notag\\
		&-2\left<\pab(\nabla_v(f^n-f^l)\cdot\dee W_t),\pab(f^n-f^l)\right>_m\notag\\
		&+\sum_k\left<\pab(\sev)^2(f^n-f^l),\pab(f^n-f^l)\right>_m\dt\notag\\
		&+\sum_k\|\pab(\sev (f^n-f^l))\|_{L_m^2}^2\dt.
	\end{align}
We will control the above for $|\alpha|+|\beta|=\sigma$. Of course, the linear terms are treated in the same way as in the estimates of Section \ref{sec:iii}. We now explain how the electric field terms are to be estimated. For $t\leq \tau_{n,l}$, we have:
\begin{align}
&\left|\left<\pab(E^n\cdot\nabla_vf^n -E^l\cdot\nabla_vf^l),\pab(f^n-f^l)\right>_m\right| \notag\\
\leq& \left| \left<\pab((E^n-E^l)\cdot\nabla_v f^n), \pab(f^n-f^l)\right>_m  \right|
+ \left|\left<\pab((E^l\cdot\nabla_v(f^n-f^l)),\pab(f^n-f^l)\right>_m\right|\notag\\
\leq&C\|f^n-f^l\|_{H_m^{\sigma-1}}\|f^n\|_{H_m^{\sigma+1}}\|f^n-f^l\|_{H_m^\sigma}+C\|f^l\|_{H_m^{\sigma-1}}\|f^n-f^l\|_{H_m^{\sigma}}^2\notag\\
\leq&C\|f^n-f^l\|_{H_m^\sigma}^2 + \|f^n-f^l\|_{H_m^{\sigma-1}}^2\|f^n\|_{H_m^{\sigma+1}}^2,\label{ACL:1:e:f:e}
\end{align}
where we have used Lemma \ref{prod:rules}.
Combining our estimates from the previous section with \eqref{ACL:1:e:f:e} and the fact that we are taking $t\leq \tau_{j,k}^T$, we have:
\begin{align}
	\dee\left(\|f^n-f^l\|_{H_m^\sigma}^2\right)\leq& C\|f^n-f^l\|_{H_m^\sigma}^2\dt + C\|f^n-f^l\|_{H_m^{\sigma-1}}^2\|f^n\|_{H_m^{\sigma+1}}^2\dt\notag\\
	&-2\sum_{|\alpha|+|\beta|=\sigma}\left<\pab(\nabla_v(f^n-f^l)\cdot \dee W_t),\pab(f^n-f^l)\right>_m\notag\\
	&-2\left<\nabla_v(f^n-f^l)\cdot\dee W_t,f^n-f^l\right>_m.\label{ACL:1:en}
\end{align}
In what follows we denote
\begin{align}
\mathcal{M}_\sigma(f^n - f^l) := &-2\sum_{|\alpha|+|\beta|=\sigma}\left<\pab(\nabla_v(f^n-f^l)\cdot \dee W_t),\pab(f^n-f^l)\right>_m \notag\\ 
	&\quad -2\left<\nabla_v(f^n-f^l)\cdot\dee W_t,f^n-f^l\right>_m. \label{def:mart}
\end{align}
The estimate \eqref{ACL:1:en} would close similarly to before (i.e. by using BDG and Gr\"onwall's inequalities), save for the fact that we do not a priori know that the term $\|f^n-f^l\|_{H_m^{\sigma-1}}^2\|f^n\|_{H_m^{\sigma+1}}^2$ is in $L_\omega^1L_t^1$, so we now estimate it separately\footnote{This loss of probabilistic moments was addressed by the cutoff in the approximation scheme of Section \ref{sec:iii}. } (compare to \cite{GV14}*{Lemma 7.2}).
The stochastic product rule gives:
\begin{align}\dee(\|f^n-f^l\|_{H_m^{\sigma-1}}^2\|f^n\|_{H_m^{\sigma+1}}^2) =& \|f^n\|_{H_m^{\sigma+1}}^2\dee\|f^n-f^l\|_{H_m^{\sigma-1}}  + \|f^n-f^l\|_{H_m^{\sigma-1}}^2\dee\|f^n\|_{H_m^{\sigma+1}}^2\notag\\
	&+(\dee\|f^n-f^l\|_{H_m^{\sigma-1}}^2)(\dee\|f^n\|_{H_m^{\sigma+1}}^2).\label{ACL:1:stochprod}
\end{align}
The correction term in \eqref{ACL:1:stochprod} is:
\begin{equation}
	(\dee\|f^n-f^l\|_{H_m^{\sigma-1}}^2)(\dee\|f^n\|_{H_m^{\sigma+1}}^2) = \sum_k B_{k,\sigma-1}(f^n-f^l)B_{k,\sigma+1}(f^n)
\end{equation}
where:
\begin{equation}
	B_{k,s}(h) := \left<\sev(h),h\right>_m+\sum_{|\alpha|+|\beta|=s}\left<\pab(\sev h),\pab h\right>_m.
\end{equation}
Note that $B_{k,s}(h) \leq C\sigma_k \|e^k\|_{W^{s,\infty}}\|h\|_{H_m^s}^2,$ so:
\begin{equation}
	\dee(\|f^n-f^l\|_{H_m^{\sigma-1}}^2)\dee(\|f^n\|_{H_m^{\sigma-1}}^2) \leq C\|f^n-f^l\|_{H_m^{\sigma-1}}^2\|f^n\|_{H_m^{\sigma+1}}^2\dt.\label{ACL:1:correction}
\end{equation}
For the main terms of \eqref{ACL:1:stochprod} we have similar estimates as before. For the difference $f^n-f^l$ and for $t\leq \tau_{n,l}$ we have (recalling the definition \eqref{def:mart}): 
\begin{align}
		&\dee\|f^n-f^l\|_{H_m^{\sigma-1}}^2\notag\\
	&\leq C\|f^n-f^l\|_{H_m^{\sigma-1}}^2\dt + C\|f^n-f^l\|_{H_m^{\sigma-1}}^2\|f^n\|_{H_m^{\sigma}}^2\dt\notag\\
	&-2\sum_{|\alpha|+|\beta|=\sigma-1}\left<\pab\left[\nabla_v(f^n-f^l)\cdot \dee W_t\right], \pab(f^n-f^l)\right>_m\notag\\
	&-2\left<\nabla_v(f^n-f^l)\cdot\dee W_t, f^n-f^l\right>_m\notag\\
	\leq& C\|f^n-f^l\|_{H_m^{\sigma-1}}^2\dt + \cM_{\sigma-1}(f^n-f^l),\label{ACL:1:diff}
\end{align}
where we used \eqref{ACL:1:e:f:e} for $\sigma-1$ instead of $\sigma$ derivatives and the definition of the stopping time $\tau_n^T$. Similarly, for the norm of $f^n$ and for $t\leq \tau_{n,l}$ we have:
\begin{align}
	\dee\|f^n\|_{H_m^{\sigma+1}}^2 \leq& C\|f^n\|_{H_m^{\sigma+1}}^2\dt\notag\\
&+C\|f^n\|_{H_m^\sigma}\|f^n\|_{H_m^{\sigma+1}}^2\dt\notag\\
&-2\sum_{|\alpha|+|\beta|=\sigma+1}\left<\pab(\nabla_vf^n\cdot\dee W_t),\pab f^n\right>_m\notag\\
&-2\left<\nabla_vf^n\cdot \dee W_t, f^n\right>_m\notag\\
\leq& C\|f^n\|_{H_m^{\sigma+1}}^2\dt+\cM_{\sigma+1}(f^n),\label{ACL:1:fn}
\end{align}
where we again used the definition of the stopping time $\tau_n^T$.
Now, plugging \eqref{ACL:1:correction}, \eqref{ACL:1:diff} \eqref{ACL:1:fn} into \eqref{ACL:1:stochprod}, we obtain:
\begin{align}
	&\dee(\|f^n-f^l\|_{H_m^{\sigma-1}}^2\|f^n\|_{H_m^{\sigma+1}}^2) \leq C\|f^n-f^l\|_{H_m^{\sigma-1}}^2\|f^n\|_{H_m^{\sigma+1}}^2\dt \notag\\
	&+ \cM_{\sigma+1}(f^n)\|f^n-f^l\|_{H_m^{\sigma-1}}^2 + \cM_{\sigma-1}(f^n-f^l)\|f^n\|_{H_m^{\sigma+1}}^2, \label{ACL:1:prod:est}
\end{align}
Integrating \eqref{ACL:1:prod:est} in time and using the BDG inequality, we obtain:
\begin{align*}
	\E\sup_{t'\leq t\wedge \tau_{n,l}}\left( \|f^n-f^l\|_{H_m^{\sigma-1}}^2\|f^n\|_{H_m^{\sigma+1}}^2\right) \leq& \E\left(\|f_0^n-f_0^l\|_{H_m^{\sigma-1}}^2\|f_0^n\|_{H_m^{\sigma+1}}^2\right) \\
	&+C \E\int_0^{t}\sup_{t'\leq s\wedge\tau_{n,l}}\|f^n-f^l\|_{H_m^{\sigma-1}}^2\|f^n\|_{H_m^{\sigma+1}}^2\ds \\
	&+C\E\left(\int_0^t\|f^n-f^l\|_{H_m^{\sigma-1}}^4\|f^n\|_{H_m^{\sigma+1}}^4\ds\right)^{1/2},
\end{align*}
so after rearranging and using Gr\"onwall's inequality as done previously in e.g. the proof of Lemma \ref{comp:it}, we get:
\begin{equation}
	\E\sup_{t'\leq t\wedge\tau_{n,l}}\left(\|f^n-f^l\|_{H_m^{\sigma-1}}^2\|f^n\|_{H_m^{\sigma+1}}^2\right) \leq C\E\left(\|f_0^n-f_0^l\|_{H_m^{\sigma-1}}^2\|f_0^n\|_{H_m^{\sigma+1}}^2\right)\label{prod:est:gronwall}.
\end{equation}
Now returning to \eqref{ACL:1:en}, integrating in time, using the BDG inequality, plugging in \eqref{prod:est:gronwall}, we obtain:
\begin{align*}
	\E\sup_{t'\leq t\wedge\tau_{n,l}}\|f^n-f^l\|_{H_m^{\sigma}}^2 \leq& \E\|f_0^n-f_0^l\|_{H_m^\sigma}^2 \\
	&+ C\int_0^t\E\sup_{t'\leq s\wedge\tau_{n,l}}\|f^n-f^l\|_{H_m^\sigma}^2\ds\\
	&+C\E\left(\int_0^{t\wedge\tau_{n,l}}\|f^n-f^l\|_{H_m^\sigma}^4\ds\right)^{\frac{1}{2}}\\
	&+C\E\left(\|f_0^n-f_0^l\|_{H_m^{\sigma-1}}^2\|f_0^n\|_{H_m^{\sigma+1}}^2\right).
\end{align*}
Therefore we have
\begin{align}
	\E\sup_{t'\leq \tau_{n,l}}\|f^n-f^l\|_{H_m^\sigma}^2 \leq& C\E\|f_0^n-f_0^l\|_{H_m^\sigma}^2\notag \\
	&+ C\E\left(\|f_0^n-f_0^l\|_{H_m^{\sigma-1}}^2\|f_0^n\|_{H_m^{\sigma+1}}^2\right).\label{ACL:1:fin:est}
\end{align}
Then \eqref{ACL:1} follows from \eqref{ACL:1:fin:est} and Lemma \ref{lem:RegProp} (in particular, note \eqref{ineq:RnCopt}). 

Next, we move to the proof of \eqref{ACL:2}. By It\^o's formula, we have:
\begin{align}
\dee\|f^n\|_{H_m^\sigma}^2 \leq& C\|f^n\|_{H_m^\sigma}^2\dt + \|E^n\|_{H^\sigma}^2\dt\notag\\
&-2\sum_{|\alpha|+|\beta|=\sigma}\left<\pab(\nabla_vf^n\cdot\dee W_t),\pab f^n\right>_m\notag\\
&-2\left<\nabla_vf^n\cdot\dee W_t,f^n\right>_m. \label{ACL:2:Ito}
\end{align}
Let us denote
\begin{equation*}
	\cM_{\sigma}(f^n) := 2\sum_{|\alpha|+|\beta|=\sigma}\left<\pab(\nabla_vf^n)\cdot\dee W_t,\pab f^n\right>_m + 2\left<\nabla_vf^n\cdot\dee W_t,f^n\right>_m
\end{equation*} 
so that for $t\leq \tau_n^T\wedge\epsilon$, after integrating in time, \eqref{ACL:2:Ito} gives:
\begin{align}
	\|f^n(t)\|_{H_m^\sigma}^2 \leq \|f_0^n\|_{H_m^\sigma}^2 + C\int_0^t\|f^n(s)\|_{H_m^\sigma}^2\ds + \int_0^t\cM^n(s). 
\end{align}
Therefore, by Chebyshev's inequality for the usual deterministic integral and Doob's ienquality for the martingale, we have 
\begin{align}
	\P(\sup_{t'\leq \tau_n^T\wedge\epsilon}\|f^n\|_{H_m^\sigma}^2 > \|f_0^n\|_{H_m^\sigma}^2 + 1) \leq& \P\left(C\int_0^{\tau_n^T\wedge\epsilon}\|f^n(s)\|_{H_m^\sigma}^2\ds > \frac{1}{2}\right) \notag\\
	+&\P\left(\sup_{t'\leq \tau_n^T \wedge\epsilon}\left|\int_0^{t'}\cM^n(s)\right|>\frac{1}{2}\right)\notag\\
	\leq& C\E\int_0^{\tau_n^T\wedge\epsilon} \|f^n(s)\|_{H_m^\sigma}^2 \ds\notag\\
	&+C\E\int_0^{\tau_n^T\wedge\epsilon}\|f^n(s)\|_{H_m^\sigma}^4\ds\notag\\
	\leq& C\epsilon \E\|f_0^n\|_{H_m^\sigma}^2\notag\\
	\leq& CM\epsilon;
\end{align}
note we also used that $t\leq \tau_n^T $ implies $\|f^n(t)\|_{H_m^\sigma}\leq C$ for a constant $C>0$ that depends on the size of the initial data $f_0$ \emph{uniformly} in $n$, since $\|f_0^n\|_{H_m^\sigma}\leq C\|f_0\|_{H_m^\sigma}$ independently of $n$.
Taking $\epsilon \to 0$, we obtain \eqref{ACL:2}.
\end{proof}

Combining Lemmas \ref{lemma:ACL} and \ref{ACL:ver}, we obtain the existence of a local strong solutions to \eqref{SVPFP} when $\|f_0\|_{H_m^\sigma}\leq M<\infty$ almost surely.
A splitting of the general random initial condition similar to the one in Lemma \ref{lem:locHiRegSVPFP} can now provide a local solution whenever $f_0$ is $\F_0$-measurable with $\|f_0\|_{H_m^\sigma}<\infty$ $\P$--a.s..
Specifically, since $f_0 = \sum_{M=0}^\infty \mathbbm{1}_{M\leq \|f_0\|_{H_m^\sigma}< M+1} f_0,$ each component $f_{0,M} := \mathbbm{1}_{M\leq \|f_0\|_{H_m^\sigma}<M+1}f_0$ generates a local strong solution $(f_M,\tau_M)$ to \eqref{SVPFP} and we re-construct the full $f$ and $\tau$ using
\begin{equation}
	f = \sum_{M=0}^\infty \mathbbm{1}_{M\leq \|f_0\|_{H_m^\sigma}<M+1}f_M,
\end{equation}
and
\begin{equation}
	\tau = \sum_{M=0}^\infty \mathbbm{1}_{M\leq \|f_0\|_{H_m^\sigma}<M+1}\tau_M. 
\end{equation}
This completes the proof of Theorem \ref{thm:main}. 

\section{Hypoelliptic regularization for Vlasov-Poisson-Fokker-Planck}  \label{sec:HR}
Theorem \ref{thm:HR} follows by a priori regularization estimates of \eqref{SVPFP:reg}, specifically, it suffices to prove that solutions to \eqref{SVPFP:reg} are almost-surely $C^\infty_{x,v}$ for $t > 0$.

We first prove that if $f_0 \in H^\sigma_m$, then the solution lies in $f(t) \in H^{\sigma+1}_m$ for $t > 0$ (with size depending only on the $H^\sigma$ norm of the initial condition).
As mentioned in Section \ref{sec:intro}, this hypoelliptic regularization is proved using a time-weighted variation of the classical hypocoercive energy functional for the kinetic Fokker--Planck equation (see \cite{dric2009hypocoercivity}).
For the linear case, a related hypoelliptic regularization estimate can be found in \cite{de2018invariant}.
Taking the standard energy from \cite{dric2009hypocoercivity} and scaling derivatives with the powers of $t$ expected from known hypoelliptic regularization estimates (alternatively, one can deduce them from scaling arguments; see e.g. \cite{bouchut1993existence}) 
we have 
\begin{align*}
	\mathcal{E}_1[t,f] = \norm{f(t)}_{L_m^2}^2 + at\norm{\grad_vf(t)}_{L_m^2}^2 + bt^2\brak{\grad_vf(t),\grad_xf(t)}_m + ct^3\norm{\grad_xf(t)}_{L_m^2}^2. 
\end{align*}
For $H^\sigma_m$ estimates we hence define 
\begin{equation}
	\mathcal{E}_{\sigma}[t,f] := \sum_{0\leq q \leq \sigma} \mathcal{E}_{1,m}[t,\grad_x^{\sigma-q}\grad_v^q f]. 
\end{equation}
The constants are chosen (indepedent of $\sigma$) such that $0 < c < b < a$ and  $b^2 < \sqrt{ac}$ so that 
\begin{align*}
\mathcal{E}_\sigma \approx \norm{f}_{H^\sigma_m} + t \norm{\grad_v f}_{H^\sigma_m}^2  + t^3 \norm{\grad_x f}_{H^\sigma_m}^2.  
\end{align*}
The parameters $a,b,c$ are chosen more specifically to satisfy: for some sufficiently small $\varepsilon \ll 1$ we require
	\begin{equation}
		\begin{cases}
			1 \geq \frac{a}{\varepsilon} \geq \frac{b}{\varepsilon^2} \geq \frac{c}{\varepsilon^3}\\
			a \leq \varepsilon \sqrt{1\cdot b}, \quad b \leq \varepsilon \sqrt{a\cdot c}.
		\end{cases}\label{constants:hypo}
	\end{equation}
We recall the proof that such $a,b,c$ exist in Lemma \ref{constants:hypo:proof} below.
Note that these conditions imply $bt^2 \leq \varepsilon a t + \varepsilon c t^3$, a fact we use repeatedly below. 

Hence, an estimate on $\mathcal{E}_\sigma$ in terms of $\norm{f_0}_{H^\sigma_m}$ implies the desired regularization estimates (along with some more quantitative information that we will not directly use here). 
The main result of this Section is the following.  
\begin{proposition} \label{prop:reg}
Let $f_0$ be a $\mathcal{F}_0$-measurable initial data and suppose that for all $p < \infty$, $\exists M_p > 0$ such that for some $\sigma>\frac{d}{2}+1$ we have:  
\begin{align}
\EE\norm{f_0}_{H^\sigma_m}^p < M_p.  \label{fin:mom:HR}
\end{align}
	Let $R < \infty$, let $f$ be the unique pathwise solution of \eqref{SVPFP:reg}.

	Then  $\exists T > 0$ depending only on $\sigma$ such that there holds for all $p < \infty$,  
	\begin{align*}
		\E\left(\sup_{0 < t < T} \mathcal{E}_{\sigma}[t,f(t)]\right)^p \leq C(R,p,M_2,M_3,\ldots).
	\end{align*}
	Therefore, almost surely $f(t) \in H^{\sigma+1}_m$ for all $0 < t < T$. 
\end{proposition}
Proposition \ref{prop:reg} implies a corresponding instantaneous regularization for the maximal pathwise solution of \eqref{SVPFP}.
Once the above proposition is proved, one may simply iterate it, observing that for all $\delta > 0$, $f(\delta)$ is an $\mathcal{F}_\delta$-measurable random variable with
\begin{align*}
	\EE \norm{f(\delta)}_{H^{\sigma+1}_m}^p < \infty. 
\end{align*}
Therefore, we may apply Proposition \ref{prop:reg} to the initial data $f(\delta)$ with $\sigma \mapsto \sigma+1$.
Finally, similar to the proof of Lemma \ref{lem:locHiRegSVPFP}, a simple cutting procedure can be applied to remove the moment constraint on the initial condition. 
Hence, to prove Theorem \ref{thm:HR}, it suffices to prove Proposition \ref{prop:reg}. 
\begin{proof}
For notational simplicity, we will take $\nu = 1$ but the same arguments (up to a suitable rescaling of the coefficients $a,b,c$) apply for any $\nu>0$.
Define the dissipation rate:
	\begin{align} 
		\cd_{\sigma}(t,f(t)) & := \sum_{\abs{\alpha} + \abs{\beta} \leq \sigma} \Bigl(\norm{\grad_v\pab f(t)}_{L_m^2}^2 + at\norm{\grad_v^2 \pab f(t)}_{L_m^2}^2 \notag \\
        & \quad  + \frac{b}{2}t^2\norm{\grad_x \pab f(t)}_{L_m^2}^2 + ct^3\norm{\grad_v\grad_x \pab f(t)}_{L_m^2}^2\Bigr), 
	\end{align}
which we show arises from $\dee \mathcal{E}_\sigma$. Note that this is almost the same as the contribution from $\dee \mathcal{E}_\sigma$ that arises when the time derivative lands on the powers of $t$ in front of the higher-order terms.
In order to reduce some of the notation in the ensuing calculation, we use $B(h,g)$ to denote an $L^2_m$-bounded bilinear form, the exact form of which is irrelevant, i.e, a form which is linear in both arguments and such that for any $h,g \in L^2_m$ 
\begin{align*}
\norm{B(h,g)}_{L^2_m} \lesssim \norm{h}_{L^2_m}\norm{g}_{L^2_m}. 
\end{align*}
The main step of the proof is to calculate the following 
\begin{align*}
\dee \mathcal{E}_\sigma & = \sum_{\abs{\alpha} + \abs{\beta} \leq \sigma} \dee \norm{\pab f(t)}_{L_m^2}^2 + \dee \left( at\norm{\grad_v \pab f(t)}_{L_m^2}^2 \right) \\ 
\\ & \quad + \sum_{\abs{\alpha} + \abs{\beta} \leq \sigma} \dee \left( bt^2\brak{\grad_v \pab f(t),\grad_x \pab f(t)}_m\right) + \dee \left(ct^3\norm{\grad_x \pab f(t)}_{L_m^2}^2 \right).  
\end{align*}
As in the proof of the various bounds in Sections \ref{sec:iii}--\ref{sec:iv}, we have:
	\begin{align}	\dee\norm{\pab f}_{L_{m}^2}^2 =& -2\brak{\pab(v\cdot\grad_x f),\pab f}_m\dt\notag\\
		&+2\brak{\Delta_v\pab f, \pab f}_m\dt\notag\\
		&+2\brak{\pab\Div_v(f v), \pab f}_m\dt\notag\\
		&-2\brak{\pab(\theta_{R}(\norm{f}_{H_{m_0}^{s_0}})E\cdot\grad_v f), \pab f}_m\dt\notag\\
		&-2\brak{\pab(\grad_v f\cdot\dee W_t),\pab f}_m\notag\\
		&+\sum_k \brak{\pab((\sev)^2 f),\pab f}_m\dt\notag\\
		&-\sum_k \norm{\pab(\sev f)}_{L_m^2}^2\dt.\label{hypo:pab:est}
	\end{align}
This formula, and its straightforward variations, are then used to expand most of the terms of $\dee \mathcal{E}_\sigma$, with the exception of the cross-terms (i.e. those multiplied by $b$).
For the cross-terms we instead have 
	\begin{align}
		&\dee \brak{\pvab f, \pxab f}_m =\notag\\
		 &-\left(\brak{\pvab(v\cdot\grad_xf),\pxab f}_m +\brak{\pvab f,\pxab(v\cdot\grad_xf)}_m\right)\dt \notag\\
		&-\left( \brak{\theta_R \pvab(E\cdot\grad_vf),\pxab f}_m+\brak{\pvab f,\theta_R\pxab(E\cdot\grad_vf)}_m \right)\dt\notag\\
		&+\left(\brak{\pvab(\Delta_vf),\pxab f}_m + \brak{\pvab f, \pxab (\Delta_vf)}_m\right)\dt \notag\\
		&+\left(\brak{\pvab(\Div_v(fv)),\pxab f}_m + \brak{\pvab f,\pxab(\Div_v(fv))}_m\right)\dt\notag\\
		&-\brak{\pvab(\grad_vf\cdot\dee W_t),\pxab f}_m - \brak{\pvab f,\pxab(\grad_vf\cdot\dee W_t)}_m \notag\\
		&+ \frac{1}{2}\sum_k \brak{\pvab(\sev)^2f,\pxab f}_m\dt\notag\\
		&+\frac{1}{2}\sum_k \brak{\pvab f, \pxab(\sev)^2f }_m\dt\notag\\
		&+\sum\brak{\pvab(\sev f),\pxab(\sev f)}_m\dt\notag\\
		:=& \cT_{c,\alpha,\beta,j}(f)+\mathcal{N}_{c,\alpha,\beta,j}(f)+\cD_{c,\alpha,\beta,j}(f)+\F_{c,\alpha,\beta,j}(f)+\cM_{c,\alpha,\beta,j}(f)+\cC_{c,\alpha,\beta,j}(f),\label{hypo:crossterm}
	\end{align}
	where we abbreviated $\theta_R := \theta_R(\norm{f}_{H_{m_0}^{s_0}})$ and $\cT_{c,\alpha,\beta,j},\mathcal{N}_{c,\alpha,\beta,j},\cD_{c,\alpha,\beta,j}, \F_{c,\alpha,\beta,j},\cM_{c,\alpha,\beta,j},\cC_{c,\alpha,\beta,j}$ indicate transport, nonlinear (electric field), dissipation, friction, martingale, and (It\^o) correction contributions (which incorporate the last three terms) to the cross-terms, respectively.

\subsection*{Linear, deterministic contributions:} 
First, we collect the contributions of the linear terms, namely those that arise from the $v \cdot \grad_x$ free transport and the Fokker--Planck operator.
The main effect of these terms is to introduce the dissipation $\cd_{\sigma}$.
The calculation is standard (see \cite{dric2009hypocoercivity}) and so we omit most of the details.
We define the total contribution of the linear terms of the SPDE for $f$ to $\dee\mathcal{E}$ by:
\begin{align}
Lin(\alpha,\beta) & := a \norm{\grad_v \pab f}_m^2 + 2 b t \brak{\grad_v \pab f, \grad_x \pab f}_m + 3c t^2 \norm{\grad_x \pab f}_{L^2_m}^2 \notag \\ 
& \quad -2\brak{\pab(v\cdot\grad_x f),\pab f}_m + 2\brak{\Delta_v\pab f, \pab f}_m \notag\\
& \quad + 2\brak{\pab\Div_v(f v), \pab f}_m \notag \\
& \quad - 2at \brak{\grad_v \pab(v\cdot\grad_x f),\grad_v \pab f}_m + 2at \brak{\grad_v \pab \Delta_v f, \grad_v \pab f}_m \notag \\ 
& \quad + 2at \brak{\pab\grad_v \Div_v(f v), \grad_v \pab f}_m \notag \\
& \quad - bt^2\left( \sum_j \cT_{c,\alpha,\beta,j} + \cD_{c,\alpha,\beta,j} + \F_{c,\alpha,\beta,j} \right) \notag \\   
& \quad - 2ct^3 \brak{\grad_x \pab(v\cdot\grad_x f),\grad_x \pab f}_m\notag\\
& \quad + 2ct^3\brak{\Delta\pxab f,\pxab f}_m + 2ct^3\brak{\pxab(\Div_v(fv)),\pxab f} \label{def:Lin} 
\end{align}
For the rest of this proof, we denote $p:=|\alpha|, q := |\beta|$.
By integration by parts we may write 
	\begin{align}
		\cT_{c,\alpha,\beta,j}(f)
		=& -\brak{\partial_x^{\alpha+e_j}\partial_v^\beta (v\cdot\grad_xf), \partial_x^{\alpha}\partial_v^{\beta+e_j}f }_m\dt - \brak{\partial_x^{\alpha+e_j}\partial_v^\beta f, \partial_x^\alpha\partial_v^{\beta+e_j}(v\cdot\grad_xf)}_m\dt \notag\\
		=& -\norm{\partial_x^{\alpha+e_j}\partial_v^\beta f}_{L_m^2}^2\dt \notag\\
		&-\brak{v\cdot\grad_x(\partial_x^{\alpha+e_j}\partial_v^\beta f), \partial_x^\alpha\partial_v^{\beta+e_j}f}_m\dt - \brak{\partial_x^{\alpha+e_j}\partial_v^\beta f, v\cdot\grad_x(\partial_x^{\alpha}\partial_v^{\beta+e_j}f)}_m\notag\\
		&+\sum_{\substack{\beta' < \beta \\ |\beta'|=1 }}\nchoosek{\beta}{\beta'}\left( \brak{\partial_x^{\alpha+e_j+\beta'}\partial_v^{\beta-\beta'}f,\partial_x^{\alpha}\partial_v^{\beta+e_j}f}_m+\brak{\partial_x^{\alpha+e_j}\partial_v^\beta f , \partial_x^{\alpha+\beta'}\partial_v^{\beta-\beta'+e_j}f}_m\right)\dt\notag\\
		=&\begin{cases}
			-\norm{\partial_x^{\alpha + e_j}f}_{L_m^2}^2\dt, &  \text{ if } \beta = 0\\
			-\norm{\pxab f}_{L_m^2}^2\dt  +\sum_{q'=q-1}^q B(\grad_x^{p+1}\grad_v^{q'}f,\grad_x^{p+1}\grad_v^qf)\dt & \text{ if } |\beta| > 0,
		\end{cases} \label{cross:t}	\end{align}
where recall from above that $B(\cdot,\cdot)$ denotes an $L^2_m$ bounded bilinear form, the exact form of which is not relevant. 
The dissipation term is more easily treated, yielding 
\begin{align}
		\cD_{c,\alpha,\beta}(f) =& \brak{\pvab \Delta_v f , \pxab f }_m\dt + \brak{\pvab f, \pxab \Delta_v f }_m\dt\notag\\
		=&-2\brak{\pvab \grad_vf,\pxab \grad_vf}_m\dt\notag\\
		&+ B(\grad_x^{p}\grad_v^{q+2}f,\grad_x^{p+1}\grad_v^qf)\dt + B(\grad_x^p\grad_v^{q+1}f,\grad_x^{p+1}\grad_v^{q+1}f)\dt. \label{cross:d}
	\end{align}
The friction term can be re-arranged as follows 
\begin{align}
		\F_{c,\alpha,\beta}(f) =& \brak{\pvab(\Div_v(fv)), \pxab f}_m\dt + \brak{\pvab f, \pxab(\Div_v(fv))}_m\dt \notag\\
		=& \brak{\Div_v(\pxab f v),\pvab f}_m\dt + \brak{\pxab f, \Div_v(\pvab f v)}_m\dt\notag\\
		&+\sum_{\substack{\beta'<\beta\\ |\beta'|=1}}\brak{\partial_x^{\alpha+e_j}\partial_v^{\beta+\beta'}f, \partial_x^{\alpha}\partial_v^{\beta+e_j}f}_m + \sum_{\substack{\beta'<\beta\\ |\beta'|=1}} \brak{ \pxab f, \partial_x^{\alpha}\partial_v^{\beta+\beta'+e_j}f }_m\dt\notag\\
		=&2\brak{\Div_v(\pxab f \pvab f v), 1}_m\dt - \brak{v\cdot\grad_v(\pxab f\pvab f),1}_m\dt \notag\\
		+&\sum_{\substack{\beta'<\beta\\ |\beta'|=1}}\brak{\partial_x^{\alpha+e_j}\partial_v^{\beta+\beta'}f, \partial_x^{\alpha}\partial_v^{\beta+e_j}f}_m + \sum_{\substack{\beta'<\beta\\ |\beta'|=1}} \brak{ \pxab f, \partial_x^{\alpha}\partial_v^{\beta+\beta'+e_j}f }_m\dt.\label{cross:f}
	\end{align}
The fundamental structure of the hypocoercive norm $\mathcal{E}$ is that the $\cT$ term gives rise to the $\grad_x$ dissipation term that would otherwise be missing from the dissipation of a kinetic equation.
That is, from \eqref{cross:t}, we obtain:
	\begin{align}
		bt^2\cT_{c,\alpha,\beta,j}(f) +bt^2\norm{\pxab f}_{L_m^2}^2\dt \leq& bt^2B\left(\grad_v(\grad_x^{p+1}\grad_v^{q-1}f),\grad_v(\grad_x^{p+1}\grad_v^{q-1}f)\right)\dt\notag\\
		&+bt^2B\left(\grad_x^{p+1}\grad_v^{q-1}f,\grad_v(\grad_x^{p+1}\grad_v^{q-1}f)\right)\dt\notag\\
		\lesssim& \frac{b}{a}t\cd_\sigma\dt + bt^2\mathcal{E}_\sigma\notag\\
		\lesssim& \varepsilon t\cd_\sigma\dt + bt^2 \mathcal{E}_\sigma\dt,\label{hypo:cross:t:final}
        \end{align}
and similarly, from \eqref{cross:d} and \eqref{cross:f}: 
\begin{align}
	bt^2\cD_{c,\alpha,\beta,j} + bt^2\F_{c,\alpha,\beta,j}  \lesssim& \left(\frac{b}{a} + \frac{ba}{c}\right) t\cd_\sigma\dt + \mathcal{E}_\sigma\dt\notag\\
	\lesssim& t\cd_\sigma \dt + \mathcal{E}_\sigma \dt .\label{hypo:cross:dfc:final}
\end{align}
Putting together the negative definite terms that arise from $\mathcal{T}_{c,\alpha,\beta,j}$ and those in \eqref{def:Lin}, we obtain for short $t$ and for some $C>0$:
\begin{align*}
Lin(\alpha,\beta) \leq -\left(2 - Ct \right) \cd_\sigma + C \mathcal{E}_\sigma.
\end{align*} 
In fact this is somewhat sub-optimal, as the second term on the right-hand side above can be taken in weaker norms.
However, such refinements will be irrelevant here as we are only interested in short time regularization.

\subsection*{Nonlinear contributions:} 
Next, we collect the contributions of the nonlinear electric field.
Namely,
\begin{align*}
NL(\alpha,\beta) & := -2\brak{\pab(\theta_{R}(\norm{f}_{H_{m_0}^{s_0}})E\cdot\grad_v f), \pab f}_m \notag\\
& \quad -2 at \brak{\pab \grad_v (\theta_{R}(\norm{f}_{H_{m_0}^{s_0}})E\cdot\grad_v f), \pab \grad_v f}_m \notag\\
& \quad - b t^2 \left( \brak{\theta_R \pvab(E\cdot\grad_vf),\pxab f}_m+\brak{\pvab f,\theta_R\pxab(E\cdot\grad_vf)}_m \right) \notag \\
& \quad -2 ct^3 \brak{\pab \grad_x (\theta_{R}(\norm{f}_{H_{m_0}^{s_0}})E\cdot\grad_v f), \pab \grad_x f}_m. 
\end{align*}
We first analyze the electric field's contribution to \eqref{hypo:pab:est}, similarly to the various nonlinear estimates of Section \ref{sec:iii}:
	\begin{align}
		&\left|\brak{\pab(E\cdot\grad_vf),\pab f}_m\right| \notag\\
		&\lesssim \sum_{\substack{\gamma<\alpha \\ |\alpha-\gamma|\geq2}}\abs{\brak{\partial_x^{\alpha-\gamma} E\cdot\grad_v\partial_v^\beta\partial_x^\gamma f,\pab f}_m}\dt + \norm{E}_{W^{1,\infty}}\norm{f}_{H_m^{\sigma}}^2\notag\\
		&\leq\sum_{\substack{\gamma<\alpha\\ |\alpha-\gamma|\geq 2}}\norm{\grad_x^{|\alpha-\gamma|}E}_{L_x^{2\frac{\sigma-1}{|\alpha-\gamma|-1}}}\norm{\grad_v\partial_v^\beta\partial_x^\gamma f}_{L_{v,m}^2L_x^{2\frac{\sigma-1}{|\beta|+|\gamma|}}}\norm{\pab f}_{L_m^2}\notag\\
		&+\norm{E}_{W^{1,\infty}}\norm{f}_{H_m^{\sigma}}^2\notag\\
		\lesssim& \sum_{\substack{\gamma<\alpha\\ |\alpha-\gamma|\geq 2}}\norm{E}_{W^{1,\infty}}^{1-\theta}\norm{E}_{H^{\sigma}}^\theta \norm{\grad_v\partial_v^\beta f}_{L_m^2}^{1-\eta}\norm{f}_{H_m^{\sigma}}^{1+\eta} + \norm{E}_{W^{1,\infty}}\norm{f}_{H_m^{\sigma}}^2,\label{hypo:pab:ef}
	\end{align}
	where for each $\alpha,\beta,\gamma,$ the interpolation index $\eta$ is given as:
	\begin{align}
		\eta =& \frac{|\gamma|}{|\alpha|-1} + \frac{d}{|\alpha|-1}\left(\frac{1}{2}-\frac{|\beta|+|\gamma|}{2(\sigma-1)}\right)\notag\\
		=&\frac{|\gamma|}{|\alpha|-1} +\frac{d}{|\alpha|-1}\cdot\frac{|\alpha-\gamma|-1}{2(\sigma-1)},
	\end{align}
	Note that in the above, $\eta<1$ since $\frac{d}{2(\sigma-1)}<1$. Therefore, 
	\begin{align}
\norm{E}_{W^{1,\infty}}^{1-\theta}\norm{E}_{H^{\sigma}}^\theta \norm{\grad_v\partial_v^\beta f}_{L_m^2}^{1-\eta}\norm{f}_{H_m^{\sigma}}^{1+\eta} & \lesssim \norm{ f}_{H^\sigma_m}^2 + \left(\norm{E}_{W^{1,\infty}}^{1-\theta}\norm{E}_{H^{\sigma}}^\theta\right)^{\frac{2}{1-\eta}} \norm{f}_{H_m^{\sigma-1}}^2 \notag \\
& \lesssim \norm{ f}_{H^\sigma_m}^2 + \norm{f}_{H_m^{\sigma-1}}^{\frac{4-2\eta}{1-\eta}}. 
	\end{align}
Note that in this term, it was not necessary to make use of the dissipation, as the first term in the final inequality above is controlled by at most $H^\sigma_m$ and the second term, which is derived from the factor $\norm{E}_{W^{1,\eta}}^{1-\theta}\norm{E}_{H^\sigma}^\theta\|\grad_v\partial_v^\beta f\|_{L_m^2}^{1-\eta}$, contains at most $\sigma-1$ derivatives (since $\abs{\beta} \leq \sigma-2$ whenever $|\alpha-\gamma|\geq 2$). 

In a similar manner, the corresponding \say{second term} in the electric field contributions to the $\pxab f$ and $\pvab f$ terms of the energy will contain at most $\sigma$ derivatives. This means that all in all, for short $t$ we can bound the electric field contributions to \eqref{hypo:pab:est}, as well as those to the higher order terms in the definition of $\mathcal{E}_\sigma$, from \eqref{hypo:pab:ef}:
\begin{align}
	\abs{\brak{\pab(E\cdot\grad_vf),\pab f}_m}\lesssim \norm{f}_{H_m^{\sigma-1}}^p + \mathcal{E}_\sigma \label{hypo:n:final:1}\\
	at\abs{\brak{\pvab(E\cdot\grad_vf),\pvab f}_m} \lesssim \norm{f}_{H_m^\sigma}^p + \mathcal{E}_\sigma\label{hypo:n:final:2}\\
	ct^3\abs{\brak{\pxab(E\cdot\grad_vf),\pxab f}_m}\lesssim \norm{f}_{H_m^\sigma}^p + \mathcal{E}_\sigma, \label{hypo:n:final:3}
\end{align}
for some $p > 2$ fixed. Note that the high power of $\norm{f}_{H^\sigma_m}$ is a priori controlled (in $L_\omega^1L_{t,loc}^\infty$) by the finite $p$-th moment assumptions \eqref{fin:mom:HR}.

We now move to estimating the electric field contribution to the cross term.
First, we integrate by parts for convenience, in order to \say{symmetrize} $\mathcal{N}_{c,\alpha,\beta,j}$ up to a lower order term:
	\begin{align}
		\mathcal{N}_{c,\alpha,\beta,j} =& \brak{\pxab(E\cdot\grad_vf),\pvab f}_m + \brak{\pxab f, \pvab(E\cdot\grad_vf)}_m\notag\\
		=& -\brak{\pab(E\cdot\grad_vf),\partial_x^{\alpha+e_j}\partial_v^{\beta+e_j}f}_m + \brak{\pxab f,\pvab(E\cdot\grad_vf)}_m\notag\\
		=&\quad 2\brak{\pxab f, \pvab (E\cdot\grad_vf)}_m \notag\\
		&+ \iint \pab(E\cdot\grad_vf)\pxab f \partial_{v^j}(\brv^m)\dv\dx.\label{cross:n:1}
	\end{align}
	We analyze the first term in the final equality above:
	\begin{align}
		&\abs{\brak{\pvab(E\cdot\grad_vf)\pxab f}_m} \notag\\
		\lesssim& \sum_{\gamma\leq \alpha} \abs{\brak{\partial_x^{\alpha-\gamma}E\cdot\grad_v\partial_v^{\beta+e_j}\partial_x^{\gamma}f,\pxab f}_m}\notag\\
		&\leq \norm{E}_{L_x^\infty}\norm{\grad_v^{q+2}\grad_x^pf}_{L_m^2}\norm{\grad_v^q\grad_x^{p+1}f}_{L_m^2}\notag\\
		&+\norm{\grad_xE}_{L_x^\infty}\norm{\grad_{v}^{q+2}\grad_x^{p-1}f}_{L_m^2}\norm{\grad_x^{p+1}\grad_v^qf}_{L_m^2}\notag\\
		&+\sum_{\substack{\gamma<\alpha\\ |\alpha-\gamma|\geq 2}}\norm{\grad_x^{|\alpha-\gamma|}E}_{L_x^{2\frac{\sigma}{|\alpha-\gamma|-1}}}\norm{\grad_v^{q+2}\grad_x^{|\gamma|}f}_{L_{v,m}^2L_x^{2\frac{\sigma}{|\beta|+|\gamma|}}}\norm{\grad_v^q\grad_x^{p+1}f}_{L_m^2}.\label{cross:n:2}
	\end{align}
	To estimate each term in the summation above, we again interpolate as in the proof of Lemma \ref{comp:epsilon};
	\begin{align}
		\norm{\grad_x^{|\alpha-\gamma|}E}_{L_x^{2\frac{\sigma}{|\alpha-\gamma|-1}}} \lesssim& \norm{\grad_xE}_{L_x^\infty}^{1-\frac{|\alpha-\gamma|-1}{\sigma}}\norm{\grad_x^{\sigma+1}E}_{L_x^2}^{\frac{|\alpha-\gamma|-1}{\sigma}},\label{cross:n:i:1}
	\end{align}
	\begin{align}
		\norm{\grad_v^{q+2}\grad_x^{|\gamma|}f}_{L_{v,m}^2L_x^{2\frac{\sigma}{|\beta|+|\gamma|}}} \lesssim& \norm{\grad_v^{q+2}f}_{L_m^2}^{1-\eta}\norm{\grad_v^{q+2}\grad_x^pf}_{L_m^2}^{\eta},\label{cross:n:i:2}
	\end{align}
	where 
	\begin{align}
		\eta =& \frac{|\gamma|}{|\alpha|} + \frac{d}{|\alpha|}\left(\frac{1}{2}-\frac{|\beta|+|\gamma|}{2\sigma}\right)\notag\\
		=& \frac{|\gamma|}{|\alpha|} + \frac{d}{|\alpha|}\frac{|\alpha-\gamma|}{2\sigma},\notag
	\end{align}
which again satisfies $\eta < 1$ since $\sigma>d/2.$ With these exponents, plugging \eqref{cross:n:i:1}-\eqref{cross:n:i:2} into \eqref{cross:n:2} we get:
\begin{align}
		&\abs{\brak{\pvab(E\cdot\grad_vf),\pxab f}_m}\notag\\
		\lesssim& \norm{E}_{L_x^\infty}\norm{\grad_v^{q+2}\grad_x^pf}_{L_m^2}\norm{\grad_v^q\grad_x^{p+1}f}_{L_m^2} + \norm{\grad_xE}_{L_x^\infty}\norm{\grad_v^{q+2}\grad_x^{p-1}f}_{L_m^2}\norm{\grad_x^{p+1}\grad_v^qf}_{L_m^2}\notag\\
		&+\sum_{\substack{\gamma<\alpha\\ |\alpha-\gamma|\geq 2}} \norm{\grad_xE}_{L_x^\infty}^{\frac{|\beta|+|\gamma|+1}{\sigma}}\norm{\grad_x^{\sigma+1}E}_{L_x^2}^{\frac{|\alpha-\gamma|-1}{\sigma}}\norm{\grad_v^{q+2}f}_{L_m^2}^{1-\eta}\norm{\grad_v^{q+2}\grad_x^pf}_{L_m^2}^\eta \norm{\grad_v^q\grad_x^{p+1} f}_{L_m^2}.\label{cross:n:3}
\end{align}
The lower order term in \eqref{cross:n:1} produces a less significant contribution - as it contains a smaller number of derivatives - and hence we omit the treatment for brevity.
Therefore, since $|\alpha-\gamma|\geq 2$ implies $q+2 \leq \sigma$, we obtain from \eqref{cross:n:3} that $\exists C(\varepsilon) >0$ such that
\begin{align}
		bt^2\mathcal{N}_{c,\alpha,\beta,j} \leq \frac{\varepsilon}{10} \cd_\sigma\dt + C(\varepsilon)\mathcal{E}_\sigma\dt + C(\varepsilon)\norm{f(t)}_{H_m^\sigma}^p\dt.\label{hypo:cross:n:final}
\end{align}
This completes the required estimates on the electric field. 

\subsection*{It\^o corrections:}
Next, let us analyze the contributions of the corrections to the cross term. This equals:
\begin{align}
		\cC_{c,\alpha,\beta,j}(f) =& \frac{1}{2}\sum_k\Bigl(\brak{\pvab(\sev)^2f,\pxab f}_m\notag\\
		&+\brak{\pvab f,\pxab(\sev)^2f}_m\Bigr)\dt\notag\\
		&+\sum_k\brak{\pvab(\sev f), \pxab (\sev f)}_m\dt\notag\\
		=& \sum_{p'\leq p}\left(B(\grad_x^{p'}\grad_v^{q+2}f, \grad_x^{p+1}\grad_v^{q+1}f)+B(\grad_x^{p+1}\grad_v^{q+2}f, \grad_x^{p'}\grad_v^{q+1}f)\right)\dt\notag\\
		&+\sum_{p'\leq p}\left(B(\grad_x^{p'}\grad_v^{q+2}f,\grad_x^{p+1}\grad_v^qf ) + B(\grad_x^{p+1}\grad_v^{q+1}f, \grad_x^{p'}\grad_v^{q+1} f ) \right)\dt\notag\\
		&+ \sum_{p'\leq p}\sum_{p''\leq p+1}B(\grad_x^{p'}\grad_v^{q+2}f, \grad_x^{p''}\grad_v^{q+1}f) \dt. \label{cross:c}
\end{align}
As mentioned above, here $B$ denotes a bilinear form which is bounded on $L_m^2\times L_m^2$, the exact form of which is irrelevant.
It follows that 
\begin{align*}
b t^2 \cC_{c,\alpha,\beta,j} \lesssim (\varepsilon+t)\cd_\sigma + \mathcal{E}_\sigma. 
\end{align*}
As in \eqref{hypo:pab:est}, from similar calculations to those in the proof of \eqref{unif:bdd:it} in Lemma \ref{comp:it} for $p=2$, we have
\begin{align*}
\sum_k \brak{\pab((\sev)^2 f),\pab f}_m - \sum_k \norm{\pab(\sev f)}_{L_m^2}^2 & \lesssim \mathcal{E}_\sigma \\
at\sum_k \left\{\brak{\grad_v \pab((\sev)^2 f), \grad_v \pab f}_m - \sum_k \norm{\grad_v \pab(\sev f)}_{L_m^2}^2\right\}  & \lesssim \mathcal{E}_\sigma \\
ct^3 \sum_k\left\{ \brak{\grad_x \pab((\sev)^2 f), \grad_x \pab f}_m - \sum_k \norm{\grad_x \pab(\sev f)}_{L_m^2}^2 \right\}& \lesssim \mathcal{E}_\sigma. 
\end{align*}
This completes the necessary estimates on the It\^o correction terms. 

\subsection*{Final estimate:} 
For $t$ sufficiently small, combining the estimates on the linear terms of \eqref{hypo:pab:est} from the above arguments with \eqref{hypo:cross:t:final}, \eqref{hypo:cross:dfc:final}, \eqref{hypo:cross:n:final} and \eqref{hypo:n:final:1}-\eqref{hypo:n:final:3}, we ultimately obtain
\begin{align}
	\dee \mathcal{E}_\sigma \leq C\mathcal{E}_\sigma \dt - (2-C\varepsilon)\cd_\sigma\dt + \cM_\sigma + C\norm{f}_{H_m^\sigma}^p\dt,\label{hypo:pre:gronwall}
\end{align}
where $\cM_\sigma$ denotes all of the martingale terms:
\begin{align}
	\cM_{\sigma} =& -2\brak{\pab(\grad_v f\cdot\dee W_t),\pab f}_m\notag\\
			&-2at\sum_{j=1}^d\brak{\pvab(\grad_vf\cdot\dee W_t),\pvab f}_m\notag\\
			&-bt^2\sum_{j=1}^d\brak{\pvab(\grad_vf\cdot\dee W_t),\pxab f}_m \notag\\
	&- bt^2\sum_{j=1}^d\brak{\pvab f,\pxab(\grad_vf\cdot\dee W_t)}_m \notag\\
	&-2ct^3\sum_{j=1}^d\brak{\pxab(\grad_vf\cdot\dee W_t),\pxab f}_m.
\end{align}
Therefore, integrating in time and using the BDG inequality, we obtain 
\begin{align}
	\E\sup_{t'\leq \tau}\mathcal{E}_\sigma(t') + (2-C\varepsilon)\E\int_0^\tau\cd_\sigma(s)\ds \leq& \E\norm{f_0}_{H_m^\sigma}^2+C\E\int_0^\tau\mathcal{E}_\sigma(s)\ds + C\E\int_0^\tau\norm{f(s)}_{H_m^\sigma}^p\ds\notag\\
	&+C\E\left(\int_0^\tau(\mathcal{E}_\sigma(s)+\varepsilon\sqrt{\mathcal{E}_\sigma(s)\cd_\sigma(s)})^2\ds\right)^{\frac{1}{2}}\notag\\
	\leq&\E\norm{f_0}_{H_m^\sigma}^2+C\E\int_0^\tau\mathcal{E}_\sigma(s)\ds + C\E\int_0^\tau\norm{f(s)}_{H_m^\sigma}^p\ds\notag\\
	&+C\E\sup_{t'\leq\tau}\mathcal{E}^{\frac{1}{2}}(t')\left(\int_0^\tau(\mathcal{E}_\sigma(s)^{\frac{1}{2}}+\varepsilon\cd_\sigma(s)^{\frac{1}{2}})^2\ds\right)^{\frac{1}{2}}\notag\\
	\leq&\E\norm{f_0}_{H_m^\sigma}^2 + C\E\int_0^\tau\mathcal{E}_\sigma(s)\ds+C\E\int_0^\tau\norm{f(s)}_{H_m^\sigma}^p\ds\notag\\
	&+\frac{1}{2}\E\sup_{t'\leq\tau}\mathcal{E}_\sigma(t') + C\varepsilon^2\E\int_0^\tau\cd_\sigma(s)\ds.
\end{align}
Rearranging terms and  applying Gr\"onwall's inequality, we obtain 
\begin{equation}
	\E\sup_{t'\leq \tau}\mathcal{E}_\sigma(t') \lesssim 1.
\end{equation}
As is standard (and as in the proofs of the $p>2$ estimates \eqref{unif:bdd:it} and \eqref{unif:lim} in Section \ref{sec:iii}), a straightforward variation of the above argument extends to prove
\begin{equation}
	\E \left(\sup_{t'\leq \tau}\mathcal{E}_\sigma(t')\right)^p \lesssim 1, 
\end{equation}
completing the desired estimates. 
\end{proof} 
For the reader's convenience, we include the elementary proof that such constants $a,b,c$ as prescribed in \eqref{constants:hypo} do indeed exist. A more general situation is treated in \cite{villani2002review}*{Lemma A.16}.
\begin{lemma}\label{constants:hypo:proof}
	Let $\varepsilon>0$. Then there exist constants $a,b,c$ such that \eqref{constants:hypo} holds:
	\begin{equation*}
		\begin{cases}
			1 \geq \frac{a}{\varepsilon} \geq \frac{b}{\varepsilon^2} \geq \frac{c}{\varepsilon^3}\\
			a \leq \varepsilon \sqrt{1\cdot b}, \quad b \leq \varepsilon \sqrt{a\cdot c}
		\end{cases}
	\end{equation*}
 \end{lemma}
\begin{proof}
	Let $\vartheta >0$, and pick $m_1,m_2,m_3>0$ such that
	\begin{equation*}
		m_1 = 1, \quad m_2 \in (1,2) ,\quad m_3 \in (m_2, 2m_2-1)
	\end{equation*}
and set:
\begin{equation*}
	a = \vartheta^{m_1}=\vartheta,\quad b = \vartheta^{m_2},\quad c= \vartheta^{m_3}.
\end{equation*}
The proof of the lemma is concluded by picking $\vartheta $ sufficiently small.
\end{proof}

\bibliographystyle{siam}
\bibliography{Vlasov}

\end{document}